\documentclass{amsart}
\usepackage{latexsym}
\usepackage{amssymb}
\usepackage{amsfonts}
\usepackage{mathrsfs}
  \usepackage{tabularx}
\usepackage{amsthm,array,bm}
 \usepackage[colorlinks=false]{hyperref}
\usepackage{pgf} \usepackage{pgflibrarysnakes}
\usepackage{tikz}
\usepackage{float}
  \usepackage{booktabs}

  \usepackage{mathtools}
\usepackage{array}
\usepackage{color,graphicx}
\usepackage{pstricks}
\usepackage{amsmath}
\def\vs{\vskip.3cm}
\def\br{\mathbb R}
\def\bz{\mathbb Z}
\def\bc{\mathbb C}
\def\wt{\widetilde}

\def\id{\text{\rm Id\,}}
\def\Id{\text{\rm Id\,}}
\def\ve{\varepsilon}
\def\cV{\mathcal V}
\def\bV{\boldsymbol {\mathcal V}}

\def\noi{\noindent}

\def\vp{\varphi}
\def\re{\mathfrak {Re}}

\def\ds{\displaystyle}
 
   \def\tD{{\tilde D}}
 
   \newcommand{\abs}[1]{\left\lvert#1\right\rvert}
  
  \newcommand{\amal}[5]{#1\prescript{#4}{}\times^{#5}_{#3}#2}

\newtheorem{theorem}{Theorem}[section]
\newtheorem{proposition}[theorem]{Proposition}
\newtheorem{lemma}[theorem]{Lemma}

{\bfseries}{\rmfamily}
\newtheorem{definition}[theorem]{Definition}
\newtheorem{remark}[theorem]{Remark}
\newtheorem{example}[theorem]{Example}

\newtheorem{remark-definition}[theorem]{Remark and Definition}

\font\smc=cmcsc10

\newrgbcolor{lightblueA}{.50 1 1}
\newrgbcolor{violet}{.6 .1 .8} 
\newrgbcolor{lightyellow}{1 1 .8} 
\newrgbcolor{lightblue}{.80 1 1}
\newrgbcolor{mygreen}{0 .66 .05} 
\definecolor{mygreen}{rgb}{0,.66,.05}
\definecolor{lightyellow}{rgb}{1,1,.80}
\newrgbcolor{orange}{1 .6 0}
\newrgbcolor{GreenYellow}{.85 1 .31}
\newrgbcolor{Yellow}{1  1  0}
\newrgbcolor{Goldenrod}{1  .90  .16}
\newrgbcolor{Dandelion}{1  .71  .16}
\newrgbcolor{Apricot}{1  .68  .48}
\newrgbcolor{Peach}{1  .50  .30}
\newrgbcolor{Melon}{1  .54  .50}
\newrgbcolor{YellowOrange}{1  .58  0}
\newrgbcolor{Orange}{1  .39  .13}
\newrgbcolor{BurntOrange}{1  .49  0}
\newrgbcolor{Bittersweet}{1.  .4300  .24}
\newrgbcolor{RedOrange}{1  .23  .13}
\newrgbcolor{Mahogany}{1.  .4475  .4345}
\newrgbcolor{Maroon}{1.  .4084  .5376}
\newrgbcolor{BrickRed}{1.  .3592  .3232}
\newrgbcolor{Red}{1  0  0}
\newrgbcolor{OrangeRed}{1  0  .50}
\newrgbcolor{RubineRed}{1  0  .87}
\newrgbcolor{WildStrawberry}{1  .04  .61}
\newrgbcolor{Salmon}{1  .47  .62}
\newrgbcolor{CarnationPink}{1  .37  1}
\newrgbcolor{Magenta}{1  0  1}
\newrgbcolor{VioletRed}{1  .19  1}
\newrgbcolor{Rhodamine}{1  .18  1}
\newrgbcolor{Mulberry}{.6668  .1180  1.}
\newrgbcolor{RedViolet}{.9538  .4060  1.}
\newrgbcolor{Fuchsia}{.5676  .1628  1.}
\newrgbcolor{Lavender}{1  .52  1}
\newrgbcolor{Thistle}{.88  .41  1}
\newrgbcolor{Orchid}{.68  .36  1}
\newrgbcolor{DarkOrchid}{.60  .20  .80}
\newrgbcolor{Purple}{.55  .14  1}
\newrgbcolor{Plum}{.50  0  1}
\newrgbcolor{Violet}{.98 .15 .95}
\newrgbcolor{RoyalPurple}{.25  .10  1}
\newrgbcolor{BlueViolet}{.84  .38  .98}
\newrgbcolor{Periwinkle}{.43  .45  1}
\newrgbcolor{CadetBlue}{.38  .43  .77}
\newrgbcolor{CornflowerBlue}{.35  .87  1}
\newrgbcolor{MidnightBlue}{.4414  .9259  1.}
\newrgbcolor{NavyBlue}{.06  .46  1}
\newrgbcolor{RoyalBlue}{0  .50  1}
\newrgbcolor{Blue}{0  0  1}
\newrgbcolor{Cerulean}{.06  .89  1}
\newrgbcolor{Cyan}{0  1  1}
\newrgbcolor{ProcessBlue}{.04  1  1}
\newrgbcolor{SkyBlue}{.38  1  .88}
\newrgbcolor{Turquoise}{.15  1  .80}
\newrgbcolor{TealBlue}{.1572  1.  .6668}
\newrgbcolor{Aquamarine}{.18  1  .70}
\newrgbcolor{BlueGreen}{.15  1  .67}
\newrgbcolor{Emerald}{0  1  .50}
\newrgbcolor{JungleGreen}{.01  1  .48}
\newrgbcolor{SeaGreen}{.31  1  .50}
\newrgbcolor{Green}{0  1  0}
\newrgbcolor{ForestGreen}{.1992  1.  .2256}
\newrgbcolor{PineGreen}{.3100  1.  .5575}
\newrgbcolor{LimeGreen}{.50  1  0}
\newrgbcolor{YellowGreen}{.56  1  .26}
\newrgbcolor{SpringGreen}{.74  1  .24}
\newrgbcolor{OliveGreen}{.6160  1.  .4300}
\newrgbcolor{RawSienna}{.53  .28  .16}
\newrgbcolor{Sepia}{1.  .7510  .70}
\newrgbcolor{Brown}{.41  .25  .18}
\newrgbcolor{TAN}{.86  .58  .44}
\newrgbcolor{Gray}{1.  1.  1.}
\newrgbcolor{Black}{1  1  1}
\newrgbcolor{White}{1  1  1}

\begin{document}
\title[Dihedral Molecular Configurations]{Dihedral Molecular Configurations Interacting by Lennard-Jones and Coulomb Forces
}
\author{Irina Berezovik}
\email{Irina.Berezovik@utdallas.edu}
\author{Qingwen Hu}
\email{qingwen@utdallas.edu}
\author{Wieslaw Krawcewicz}
\email{wieslaw@utdallas.edu}
\address{Department of Mathematical Sciences, the University of Texas at
Dallas,
 Richardson, TX, 75080-3021, U.S.A.}

\maketitle

\begin{abstract}
In this paper,  we investigate periodic vibrations of a group of particles with a dihedral configuration in the plane governed by the Lennard-Jones and Coulomb forces.  Using the gradient equivariant degree, we provide a full topological classification of the periodic solutions with both temporal and spatial symmetries. In the process, we provide with general formulae for the spectrum of the linearized system which allows us to obtain the critical frequencies of the particle motions  which  indicate   the set of all critical periods   of small amplitude periodic solutions  emerging from a given stationary symmetric orbit of solutions.  
\end{abstract}


\section{Introduction} 
Classical forces used in molecular mechanics associated with bonding between the adjacent particles, electrostatic interactions and van der Waals forces, are  modeled using bonding by Lennard-Jones and Coulomb potentials. In a typical molecule an atom is bonded only to few of its neighbors but it also interacts with every other atom in the molecule. The famous 6-12--Lennard-Jones potential, which was  proposed in 1924 (cf. \cite{L-J}), was  found experimentally and since then it is successfully used in molecular modeling. Certainly one can expect that other types of more accurate potentials may be introduced in the future. 
\vs
To be more precise, consider  $n$ identical particles $u_j$, $j\in \{0,1,2,\dots, n-1\}=:X$, in the space $\br^3$. A   set $\mathscr B\subset X\times X$  satisfying the conditions 
(1) $(i,j)\in \mathscr B$ then $(j,i)\in \mathscr B$, and (2) $(j,j)\notin \mathscr B$, can be considered as {\it bonding set} for a specific configuration of atoms in the molecule, that means 
we suppose that the particles $u_k$ and $u_j$ are bonded if $(k,j)\in \mathscr B$.  The symmetries of a molecular bonding are reflected in the set $\mathscr B$. 
There are many  examples of symmetric atomic molecules, for example, octahedral compounds of sulfur hexafluoride SF$_6$ and molybdenum hexacarbonyl Mo(CO)$_6$, the etraphosphorus P$_4$,  a spherical fullerene molecule with the formula C$_{60}$  with icosahedral  symmetry or dihedral molecule with 2-D interactions. One can find multiple example of symmetric molecule clusters in \cite{symm}.
\vs
We describe the  molecular model considered in this paper as follows. 
Let  $u=(u_0,u_1,\dots,u_{n-1})\in \br^{2n}$ and 
$\Omega'_o:=\{ u\in\br^{2n}:  u_k\not= u_j, \text{ if }k\not=j, \text{ with }\,k,\,j=0,\,1,\,2,\,\cdots,\,n-1\}$. Define the following energy functional $\bV:\Omega'_o\to \br$ by 
\begin{equation}\label{eq:pot}
\boldsymbol {\mathcal V}(u):= \sum_{(k,j)\in \mathscr B}^n U(|u_{j}-u_k|^2) + \sum_{0\le j<k\le n-1} W(|u_j-u_k|^2),
\end{equation}
where 
\[
U(t)=t-2\sqrt t, \quad W(t)=\frac{B}{t^6}-\frac A{t^3}+\frac \sigma{\sqrt t}, \quad t>0.
\]
\vs 
The following Newtonian equation describes the interaction between these $n$-particles,
\begin{align}\label{eqn01}
\ddot{u}(t)=-\nabla \bV (u(t)),
\end{align}  

\vs

In this paper we develop a new method allowing an extraction from model \eqref{eqn01} a topological equivariant classification of $p$-periodic ($p>0$) molecular vibrations for a symmetric molecule in 2-D polygonal symmetric configuration with dihedral symmetry group. The vibrational motions, which are characteristic of all molecules can be easily detect using infrared or Raman spectroscopy, depend on the vibrational structure of electronic transitions in molecules. The vibrational motions are closely connected to the symmetric properties of $2\pi$-periodic solutions  the system 
\begin{equation}\label{eq:mol}
\begin{cases}
\ddot u(t)= -\lambda^2\nabla^2 \bV(u(t)), \quad u(t)\in \Omega'_o, \; t\in \br\\
u(0)=u(2\pi),\;\; \dot u(0)=\dot u(2\pi),
\end{cases}
\end{equation} 
where $\lambda =\frac{p}{2\pi}$, which are exactly $p$-periodic solutions to \eqref{eqn01}.

\vs
An important feature of a molecular vibration is that in general it admits spatial-temporal  symmetries (depending on the actual molecular symmetries), which called a {\it mode of vibration} (reflected in atomic motions such as stretching, bending, rocking, wagging and twisting). These modes and the corresponding vibrational frequencies are of great importance in molecular dynamics. 
It is thus desirable to distinguish periodic motions with distinct symmetric modes of vibrations.

\vs

The content of this paper can be described as follows. In section 2 we recall the basic definitions and properties related to the equivariant degree theory. In section 3, we discuss a molecular model with Lennard-Jones and Coulomb potentials for identical atoms bonded in a polygonal configuration. In subsection 3.1 we show the existence of the symmetric equilibrium $u^o$ and in subsection 3.2 we formulate the problem of finding periodic vibration as a bifurcation problem for \eqref{eq:mol} and  in subsection 3.3 we identify the $D_n$-isotypical decomposition of the phase space. In section 4, the problem \eqref{eq:mol} is reformulated as an $D_n\times O(2)$-equivariant bifurcation variational problem. The equivariant invariant $\omega(\lambda_o)$ is provided in Theorem 4.2. Section 5 is devoted to the symbolic computations of the spectrum of $\nabla^2\bV(u_o)$ (for a general potential $\bV$). In section 6 we formulate the main existence results based on the values of the equivariant invariants $\omega(\lambda_o)$ (Theorem 6.1). In section 7, we consider a concrete system \eqref{eqn01} with $D_6$-symmetries and compute several equivariant invariants iand 
how to extract the relevant equivariant information. Finally, we confirm the obtained existence results with several computer simulations (in subsection 7.2). 
\vs
\section{Preliminaries}
\subsection{Equivariant Jargon:}  
\paragraph{$G$-Actions:\label{G-actions}}
In what follows $G$ always stands for a compact Lie group and all
subgroups of $G$ are assumed to be closed. For a subgroup $H\subset G$,
denote by $N\left( H\right) $ the normalizer of $H$ in $G$, and by $W\left(
H\right) =N\left( H\right) /H$ the Weyl group of $H$ in $G$. In the case
when we are dealing with different Lie groups, we also write $N_{G}\left(
H\right) $ ($W_{G}\left( H\right) $, respectively) instead of $N\left( H\right) $
($W\left( H\right) $, respectively). We denote by $\left( H\right) $ the conjugacy class of $H$ in $G$ and
define the following notations: 
\begin{align*}
\Phi \left( G\right) & :=\left\{ \left( H\right) :H\;\;\text{is a subgroup
of }\;G\right\} , \\
\Phi _{n}\left( G\right) & :=\left\{ \left( H\right) \in \Phi \left(
G\right) :\text{\textrm{dim\,}} W\left( H\right) =n\right\} .
\end{align*}%
The set $\Phi \left( G\right) $ has a natural partial order defined by 
\begin{equation}
\left( H\right) \leq \left( K\right) \;\;\Longleftrightarrow \;\;\exists
_{g\in G}\;\;gHg^{-1}\subset K.  \label{eq:partial}
\end{equation}%
For a $G$-space $X$ 
and $x\in X$, we define
\begin{align*}
G_{x} & :=\left\{ g\in G:\;gx=x\right\}, \text{the isotropy of $x$}; \\ 
\left( G_{x}\right)  & :=\left\{ H\subset G:\;\exists _{g\in G}\;\;G_{x}=g^{-1}Hg\right\},   \text{the orbit type of $x$ in $X$}; \\ 
G\left( x\right)& :=\left\{ gx:\;g\in G\right\}, \;\;\text{the orbit of $x$}.
\end{align*}
Moreover, for a subgroup $H\subset G$, we use the following
notations:
\begin{align*}
X_{H} & :=\left\{ x\in X:\;G_{x}=H\right\} ; \\ 
X^{H} & :=\left\{ x\in X:\;G_{x}\supset H\right\} ; \\ 
X_{(H)}& :=\left\{ x\in X:\;(G_{x})=(H)\right\} ; \\ 
X^{(H)} & :=\left\{ x\in X:\;(G_{x})\geq (H)\right\} .
\end{align*}
 The orbit space for a $G$-space $X$ will be denoted by $X/G$ and for
the space $G$ by $G\backslash X$.

\vs
\paragraph{Isotypical Decomposition of Finite-Dimensional Representations:
\label{subsec:G-represent}}

As  any compact Lie group admits only countably many
non-equivalent real (complex, respectively) irreducible representations. Given a
compact Lie group $G $, we assume that we have a complete list of its all real
(complex, respectively) irreducible representations, denoted $\mathcal{V}_{i}$, $i=0, $ $1,$ $\ldots $ ($\mathcal{U}_{j}$, $j=0,$ $1,$ $\ldots $, respectively). We refer to \cite{AED} for examples of such lists and the related notations. 
\vs

Let $V$ ($U$, respectively) be a finite-dimensional real (complex, respectively) $\Gamma $-representation.  Without loss of generality, $V$ ($U$, respectively) can be
assumed to be orthogonal (unitary, respectively). Then, $V$ ($U$, respectively) decomposes
into the direct sum of $G $-invariant subspaces
\begin{equation}
V=V_{0}\oplus V_{1}\oplus \dots \oplus V_{r}\text{,}  \label{eq:Giso}
\end{equation}%
\begin{equation}
\text{(}U=U_{0}\oplus U_{1}\oplus \dots \oplus U_{s}\text{, respectively),}
\label{eq:Giso-comp}
\end{equation}%
which is called the $G$\textit{-}\emph{isotypical decomposition of }$V$ ($U$, respectively), where each isotypical component $V_{i}$ (resp. $U_{j}$) is \emph{modeled} on the irreducible $G$-representation $\mathcal{V}_{i}$, $i=0,$ $1,$ $\dots ,$ $r$, ($\mathcal{U}_{j}$, $j=0,$ $1,$ $\dots ,$ $s$, respectively), i.e. $V_{i}$ ($U_{j}$, respectively) contains all the irreducible subrepresentations of $V$ ($U$, respectively) which are equivalent to $\mathcal{V}_{i}$ ($\mathcal{U}_{j}$, respectively).
\vs
 \subsection{Gradient $G$-Equivariant Degree}
\paragraph{Euler Ring and Burnside Ring:\label{subsect:Euler}}

\begin{definition}
\label{def:EulerRing} (cf. \cite{tD}) Let
$
U\left( G\right) :={\mathbb{Z}}\left[ \Phi \left( G\right) \right]
$
denote the free $\mathbb{Z}$-module generated by $\Phi (G)$.
Define a ring multiplication on generators $\left( H\right) $, $\left(
K\right) \in \Phi \left( G\right) $ as follows: 
\begin{equation}
\left( H\right) \ast \left( K\right) =\sum_{\left( L\right) \in \Phi \left(
G\right) }n_{L}\left( L\right) ,  \label{eq:Euler-mult}
\end{equation}
where 
\begin{equation}
n_{L}:=\chi _{c}\left( \left( G/H\times G/K\right) _{L}/N\left( L\right)
\right),  \label{eq:Euler-coeff}
\end{equation}%
for $\chi _{c}$ being the Euler characteristic taken in Alexander-Spanier
cohomology with compact support (cf. \cite{Spa}). The $\mathbb{Z}
$-module $U\left( G\right) $ equipped with the multiplication \eqref{eq:Euler-mult}, \eqref{eq:Euler-coeff} is a ring called the {\it  Euler
ring} of the group $G$ (cf. \cite{BtD})
\end{definition}
\vs
 The ${\mathbb{Z}}$-module $A\left( G\right) =A_{0}\left( G\right):={\mathbb{Z}}\left[ \Phi _{0}\left( G\right) \right] $ equipped with a
similar multiplication as in $U\left( G\right) $ but restricted only to
generators from $\Phi _{0}\left( G\right) $, is called a \emph{Burnside ring}. That is,  for $\left( H\right) $, $\left( K\right), \left( L\right)  \in \Phi _{0}\left( G\right) $ 
\begin{equation*}
\left( H\right) \cdot \left( K\right) =\sum_{\left( L\right) }n_{L}\left(
L\right),
\end{equation*}%
where $n_{L}=\chi\left( \left( G/H\times G/K\right) _{L}/N\left( L\right)
\right) =\left\vert \left( G/H\times G/K\right) _{L}/N\left( L\right)
\right\vert $ and $\chi $ stands for the usual Euler characteristic. In
this case, we have 
\begin{equation}
n_{L}=\frac{n\left( L,K\right) \left\vert W\left( K\right) \right\vert
n\left( L,\text{ }H\right) \left\vert W\left( H\right) \right\vert
-\sum_{\left( \widetilde{L}\right) >\left( L\right) }n\left( L,\; \widetilde{L}\right) n_{\widetilde{L}}\left\vert W\left( \widetilde{L}\right) \right\vert }{\left\vert W\left( L\right) \right\vert },
\label{eq:rec-coef}
\end{equation}
where 
\[
n(L,K)=\left| \frac{N(L,K)}{N(K)} \right|, \quad N(L,K):=\{g\in G: gLg^{-1}\subset K\},
\]
and $\left( H\right) ,$ $\left( K\right) ,$ $\left( L\right) $, $\left( 
\widetilde{L}\right) $ are taken from $\Phi _{0}\left( G\right) $.\smallskip
\vs Notice  that $A\left( G\right) $ is  a ${\mathbb{Z}}$-submodule of $U\left( G\right) 
$, but not a subring. Define $\pi _{0}:U\left( G\right) \rightarrow A\left( G\right) $
on generators $\left( H\right) \in \Phi \left( G\right) $ by 
\begin{equation}
\pi _{0}\left( \left( H\right) \right) =%
\begin{cases}
\left( H\right) & \text{ if }\;\left( H\right) \in \Phi _{0}\left( G\right) ,
\\ 
0 & \text{ otherwise.}%
\end{cases}
\label{eq:pi_0-homomorphism}
\end{equation}%
Then we have,
\begin{lemma}
\label{lem:pi_0-homomorphism} (cf. \cite{BKR}) The map $%
\pi _{0}$ defined by $(\mathrm{\ref{eq:pi_0-homomorphism}})$ is a ring
homomorphism, that is, 
\begin{equation*}
\pi _{0}\left( \left( H\right) \ast \left( K\right) \right) =\pi _{0}\left(
\left( H\right) \right) \cdot \pi _{0}\left( \left( K\right) \right),
\end{equation*}where $\left( H\right) ,\text{ }\left( K\right) \in \Phi \left( G\right).$
\end{lemma}
Lemma \ref {lem:pi_0-homomorphism} allows us to use  Burnside ring multiplication structure in $A\left( G\right) $ to partially
 describe the Euler ring multiplication structure in $U\left( G\right) $.
\vs
\paragraph{\bf $G$-equivariant Gradient Degree $\nabla_G\text{\rm -deg}$:}
Assume that $G$ is a compact Lie group. Denote by $\mathcal M^G_{\nabla}$ the set of all admissible pairs $(\nabla \varphi,\Omega)$.

\vs
\begin{definition}\rm\label{def:gene}
A $G$-gradient $\Omega$-admissible map $f:=\nabla \vp$ is called a {\it special } 
 $\Omega$-Morse function if 
 \begin{itemize}
 \item[(i)] $f|_{\Omega}$ is of class $C^1$;
\item[(ii)] $f^{-1}(0)\cap \Omega$ is composed of regular zero orbits;
\item[(iii)] for each $(H)$ with $f^{-1}(0)\cap\Omega_{(H)}\ne\emptyset$, 
there exists a tubular neighborhood $\mathcal N(U,\ve)$ such that $f$ is 
$(H)$-normal on $\mathcal N(U,\ve)$.
\end{itemize}
\end{definition}
\vs

 We have,
\begin{theorem}\label{thm:Ggrad-properties}(cf. \cite{Geba})
There exists a unique map $\nabla_G\text{\rm -deg\,}:\mathcal M_\nabla^G\to U(G)$, which assigns to every $(\nabla \varphi,\Omega) \in \mathcal M^G_{\nabla}$ an element $\nabla_G\text{\rm -deg\,}(\nabla \varphi,\Omega)\in U(G)$, called the {\it $G$-gradient degree} of $\nabla \varphi$ on $\Omega$,
\begin{equation}\label{eq:grad-deg}
\nabla_G\text{\rm -deg\,}(\nabla \varphi,\Omega)=\sum_{(H_i)\in \Phi(\Gamma)} n_{H_i}(H_i)=n_{H_1}(H_1)+\dots + n_{H_m}(H_m),
\end{equation}satisfying  the following
properties:
\begin{itemize}
\item[($\nabla$1)\hspace{-.095cm}]   {\bf (Existence)}   If
$\nabla_G\text{\rm -deg\,}(\nabla \varphi,\Omega)\not=0$,  that is, there is in \eqref{eq:grad-deg} a non-zero coefficient $n_{H_i}$, then there exists
${x\in\Omega  }$ such that $\nabla \varphi (x)=0$ and $(G_x)\geq (H_i)$.

\item[($\nabla$2)\hspace{-.095cm}]   {\bf (Additivity)}    Let $\Omega  _1$ and $\Omega  _2$ be
two disjoint open $G$-invariant subsets of $\Omega  $ such that $(\nabla\varphi)^{-1}(0)\cap
\Omega  \subset \Omega  _1\cup \Omega  _2.$  Then,
$$
\nabla_G\text{\rm -deg\,}(\nabla \varphi,\Omega)=\nabla_G\text{\rm -deg\,}(\nabla \varphi,\Omega_1) + \nabla_G\text{\rm -deg\,}(\nabla \varphi,\Omega_2).
$$

\item[($\nabla$3)\hspace{-.095cm}]   {\bf (Homotopy)}   If $\nabla_v\Psi:[0,1]\times V\to V$ is
  a $G$-gradient
$\Omega$-admissible homotopy, then
$$
\nabla_G\text{\rm -deg\,}(\nabla_v\Psi(t,\cdot),\Omega)=\text{ \it a constant}.$$

\item[($\nabla$4)\hspace{-.095cm}]   {\bf (Normalization)} Let $\varphi\in  C^2_G(V,\mathbb R)$ be a special $\Omega$-Morse function such that $(\nabla\varphi)^{-1}(0)\cap \Omega=G(v_0)$ and $G_{v_0}=H$. Then,
$$
\nabla_G\text{\rm -deg\,}(\nabla \varphi,\Omega)= (-1)^{\mathrm{m}^-(\nabla^2\varphi(v_0))}\cdot (H),
$$
where ``$\mathrm{m}^-(\cdot)$'' stands for the total dimension of  eigenspaces
for negative eigenvalues of a symmetric matrix.
\item[($\nabla$5)\hspace{-.095cm}]   {\bf (Multiplicativity)} For all $(\nabla\varphi_1,\Omega_1)$, $ (\nabla\varphi_2,\Omega_2) \in \mathcal M^G_\nabla$,
$$
\nabla_G\text{\rm -deg\,}(\nabla \varphi_1\times\nabla\varphi_2,\Omega_1\times\Omega_2)=\nabla_G\text{\rm -deg\,}(\nabla \varphi_1,\Omega_1)\ast \nabla_G\text{\rm -deg\,}(\nabla \varphi_2,\Omega_2)
$$
where the multiplication `$\ast$' is taken in the Euler ring $U(G)$.

\item[($\nabla$6)\hspace{-.095cm}]   {\bf (Suspension)}  {\it If $W$ is an orthogonal $G$-representation and $\mathcal B$ an open bounded invariant neighborhood of $0\in W$, then}
$$
\nabla_G\text{\rm -deg\,}(\nabla \varphi\times \mbox{\rm Id}_W,\Omega\times \mathcal B) = \nabla_G\text{\rm -deg\,}(\nabla \varphi,\Omega).
$$

\item[($\nabla$7)\hspace{-.095cm}] {\bf (Hopf Property)} Assume that $B(V)$ is the unit ball of an
orthogonal $\Gamma$-representa\-tion $V$ and for $(\nabla\varphi_1,B(V)),(\nabla\varphi_2,B(V)) \in
\mathcal M^{G}_{\nabla}$, one has
\[ \nabla_G\text{\rm -deg\,}(\nabla\varphi_1,B(V)) =  \nabla_G\text{\rm -deg\,}(\nabla\varphi_2,B(V)).\]
Then $\nabla\varphi_1$ and $\nabla\varphi_2$ are
$G$-gradient $B(V)$-admissible homotopic.
\end{itemize}
\end{theorem}
\vskip.2cm
\paragraph{\bf Computations of the Gradient $G$-Equivariant Degree:} 
Consider a symmetric $G$-equivariant linear isomorphism $T:V\to V$, where $V$ is an orthogonal $G$-representation, that is,  $T=\nabla \varphi$ for $\varphi(v)=\frac 12 (Tv\bullet T)$, $v\in V$, where ``$\bullet$'' stands for the inner product. We will show how to  compute $\nabla_G\text{-deg\,}(T,B(V))$. Consider the $G$-isotypical decomposition \eqref{eq:Giso} of $V$ and put
\[
T_i:=T|_{V_i}:V_i\to V_i,\quad T_{j,l}:=T|_{V_{j,l}}:V_{j,l}\to V_{j,l}.
\]
Then, by the Multiplicativity property ($\nabla$5),
\begin{equation}\label{eq:deg-Lin-decoGrad}
\nabla_G\mbox{-deg}(T,B(V))=\prod_{i}^r \nabla_G\mbox{-deg}(T_i,B(V_i))*\prod_{j,l} \nabla_G\mbox{-deg}(T_{j,l},B(V_{j,l})).
\end{equation}
Take $\mu\in \sigma_-(T)$, where $\sigma_-(T)$ stands for the negative spectrum of $T$,  and consider the corresponding eigenspace  $ E(\mu):= \ker (T-\mu\mbox{Id})$. Define the numbers $m_i(\mu)$  and $m_{j,l}(\mu)$ by
\begin{equation}\label{eq:m_j(mu)-gra}
m_i(\mu):= \dim \left(E(\mu)\cap V_i \right)/\dim \mathcal V_i, \quad m_{j,l}(\mu):= \dim \left(E(\mu)\cap V_{j,l} \right)/\dim \mathcal V_{j,l}.
\end{equation}
We also define the   basic gradient degrees by
\begin{equation}\label{eq:basicGrad-deg0}
\mbox{Deg}_{\mathcal V_i}:=\nabla_G\mbox{-deg}(-\mbox{Id\,},B(\mathcal V_i)), \quad \mbox{Deg}_{\mathcal V_{j,l}}:=\nabla_G\mbox{-deg}(-\mbox{Id\,},B(\mathcal V_{j,l})).
\end{equation}
We  have that,
\begin{itemize}
\item[(i)] $\mbox{\rm Deg}_{\mathcal V_i} = \deg_{\mathcal V_i}$;
\item[(ii)] $\mbox{\rm Deg}_{\mathcal V_{j,l}}=(G)-\deg_{\mathcal V_{j,l}}$,
\end{itemize}
\noindent
where $\deg_{\mathcal V_i}$  is the  basic  $G$-equivariant degree without free parameter and  $\deg_{\mathcal V_{j,l}}$ is the    basic $G$-equivariant basic twisted degree, which was introduced in \cite{AED}. The basic degree 
\[
\mbox{\rm deg}_{\mathcal V_i} =(G)+n_{L_1}(L_1)+\dots +n_{L_n}(L_n),\quad L_0:=G,
\]
can be computed from the recurrence formula
\begin{equation}\label{eq:coeff-jo}
n_{L_{k}}=\frac{(-1)^{n_{k}}-\sum_{L_{k}<L_l} n(L_{k},L_k)\cdot  n_{L_l}\cdot |W(L_l)|}{|W(L_{k})|},
\end{equation}
and the twisted degree 
\[
\deg_{\mathcal V_{j,l}}=n_{H_1}(H_1)+n_{H_2}(H_2)+\dots + n_{H_m}(H_m),
\]
 can be computed from the recurrence formula
\begin{equation}\label{eq:bdeg-nL}
n_{H_k}=\frac{\frac 12 \dim\, \mathcal V_{j,l}^{H_k}-\sum_{H_k<H_s}n_{H_s}\, n(H_k,H_s)\, |W(H_s)/S^1|}{\left|
\frac{W(H_k)}{S^1}\right|}.
\end{equation}
One can also find in \cite{AED}  complete lists of these basic degrees for several groups $G=\Gamma\times S^1$.  Then, by using the properties of gradient $G$-equivariant degree, one can  establish
\begin{equation}\label{eq:lin-GdegGrad}
\nabla_G\text{\rm -deg}(T,B(V))=
\prod_{\mu\in\sigma_-(T)}\prod_{i}^r\left({\mbox{\rm deg}_i}\right)^{m_i(\mu)}*\prod_{j,l}\left((G)-\mbox{deg}_{j,l}\right)^{m_{j,l}(\mu)}.
\end{equation}
\vs

\subsection{Gradient Degree on the Slice}
Let $G$ be a compact Lie 
group and $\mathscr{H}$ be a smooth Hilbert $G$-representation (that is, $T:G\to O(\mathscr H)$ is a smooth map).  Let  $\varphi :
\mathscr{H}\rightarrow \mathbb{R}$ be a continuously differentiable $G$-invariant functional. Then the gradient
\begin{equation*}
\nabla \varphi :\mathscr{H}\rightarrow \mathscr{H}
\end{equation*}
is a well-defined $G$-equivariant operator.
Let $u_o\in \mathscr H$  and  put  $H:=G_{u_o}$.  Since the $G$-action on $\mathscr H$ is smooth, the orbit  $G(u_o)$ is a smooth submanifold of $\mathscr H$. 
Denote by  $S_o\subset \mathscr H$  the slice to the orbit  $G(u_o)$ at $u_o$. Denote by  $V_o:=\tau_{u_o} G(u_o)$ the tangent space to $G(u_o)$ at $u_o$. Then clearly, 
 $S_o=V_o^\perp$ and $S_o$ is a smooth Hilbert $H$-representation. 
 \vs 
 \begin{theorem} {\smc (Slice  Principle)} \label{thm:SCP} 
 Let $\mathscr{E}$ be an orthogonal $G$-representation, $\varphi :
\mathscr{H}\rightarrow \mathbb{R}$ be a continuously differentiable $G$-invariant functional, $u_o\in \mathscr H$ and $G(u_o)$ be an isolated critical orbit of $\vp$ such that   $H:=G_{u_o}$. Let  $S_o$ be the slice to the orbit  $G(u_o)$ at $u_o$  and $\mathcal U$ an isolated tubular neighborhood of $G(u_o)$.   Define  $\vp_o:S_o\to \br$ by $\vp_o(v)=\vp(u_o+v)$, $v\in S_o$. 
Then
\begin{equation}\label{eq:SDP}
  \nabla_G\text{\rm -deg}(\nabla \vp,\, \mathcal U)= \Theta ( \nabla_H\text{\rm -deg}(\nabla \vp_o,\, \mathcal U\cap S_o)),
  \end{equation}
where $\Theta:U(H)\to U(G)$ is defined on generators $\Theta(K)=(K)$, $(K)\in \Phi(H)$. 
\end{theorem} 
\begin{proof}
Notice that $\nabla \vp_o(u)=P\nabla \vp(u)$, where $P:\mathscr E\to S_o$ is an orthogonal projection. Since one can always approximate $\vp_o$ on $\mathcal U \cap S_o$  by a generic map, which can be extended equivariantly on $\mathcal U$  and, in  such a case, this extension is also generic, formula \eqref{eq:SDP} follows directly  from the definition of the gradient degree for generic maps. 
\end{proof} 
 

\subsection{Product Group $G_{1}\times G_{2}$}
Given two groups  $\mathscr G_1$ and $\mathscr G_2$, consider  the product group and $\mathscr G_1\times \mathscr G_2$. The following well-known result  (see \cite{DKY,Goursat}) provides a description of  subgroups $\mathscr H$ of the product group $\mathscr G_1\times \mathscr G_2$.

 \begin{theorem}
\label{th:prod1}
Let  $\mathscr H$ be a subgroup of the product group $\mathscr G_1\times \mathscr G_2$. Put  $ H:=\pi_1(\mathscr H)$ and $K:=\pi_2(\mathscr H)$. Then, there exist a group $L$ and two epimorphisms $\vp :H\rightarrow L$ and $\psi :K\rightarrow L$,
such that 
\begin{equation}\label{eq:product-rep}
\mathscr H=\{(h,k)\in H\times K: \vp(h)=\psi(k)\},
\end{equation}
 In this case,  we will use the notation
  \begin{align}
    \label{eq:subg_notation}
    \mathscr H=:\amal{H}{K}{L}{\varphi}{\psi}.
  \end{align}
\end{theorem}
\vs The conjugacy classes of subgroups  of $\mathscr G_1\times \mathscr G_2$, one needs the following statement (see \cite{DKY}).
 \vs

  \begin{proposition}\label{prop:conj-classes}
  Let $\mathcal G_1$ and $\mathcal G_2$ be two groups.   
  Two subgroups $H{^\varphi \times _L^\psi}K, H'{^{\varphi'} \times _L^{\psi'}}K'$
  of $\mathcal G_1 \times \mathcal G_2$  are conjugate if and only if there exist 
  $(a,b)\in \mathcal G_1\times \mathcal G_2$  and $\alpha\in \text{\rm Aut\,}(L)$ such that the inner
  automorphisms $\mu_a: \mathcal G_1\to \mathcal G_1$ and $\mu_b:\mathcal G_2\to \mathcal G_2$ given by
  \begin{align*}
    \mu_a(g_1)=a^{-1}g_1a, \quad \mu_b(g_2)=b^{-1}g_2b,\quad g_1\in \mathcal G_1, \; g_2\in \mathcal G_2,
  \end{align*}
satisfy that $H'=\mu_a(H)$, $K'=\mu_b(K)$ and $\vp=\alpha\circ \vp'\circ \mu_a$, $\psi=\alpha\circ \psi'\circ \mu_b$.\end{proposition}
\vs
 \begin{table}[h]
    \centering
    \begin{tabular}{|rcc|rcc|rcc|}
      \toprule
      ID & $(S)$ & $\abs{W(S)}$ &
        ID & $(S)$ & $\abs{W(S)}$ &
        ID & $(S)$ & $\abs{W(S)}$ \\
      \midrule
      1 & $(\amal{\bz_1}{\bz_{n}}{}{}{})$ & $\infty$ &
        35 & $(\amal{D_6}{D_{n}}{D_{1}}{D_3}{})$ & $4$ &
        69 & $(\amal{D_6}{D_{2n}}{\bz_{2}}{\tD_3}{})$ & $2$ \\
      2 & $(\amal{D_1}{\bz_{n}}{}{}{})$ & $\infty$ &
        36 & $(\amal{D_6}{D_{n}}{D_{1}}{\bz_6}{})$ & $4$ &
        70 & $(\amal{\bz_1}{SO(2)}{}{}{})$ & $24$ \\
      3 & $(\amal{\bz_2}{\bz_{n}}{}{}{})$ & $\infty$ &
        37 & $(\amal{D_6}{D_{n}}{D_{1}}{\tD_3}{})$ & $4$ &
        71 & $(\amal{D_1}{SO(2)}{}{}{})$ & $4$ \\
      4 & $(\amal{\tD_1}{\bz_{n}}{}{}{})$ & $\infty$ &
        38 & $(\amal{D_2}{D_{2n}}{D_{2}}{\bz_1}{D_1})$ & $4$ &
        72 & $(\amal{\bz_2}{SO(2)}{}{}{})$ & $12$ \\
      5 & $(\amal{\bz_3}{\bz_{n}}{}{}{})$ & $\infty$ &
        39 & $(\amal{D_2}{D_{2n}}{D_{2}}{\bz_1}{\bz_2})$ & $4$ &
        73 & $(\amal{\tD_1}{SO(2)}{}{}{})$ & $4$ \\
      6 & $(\amal{D_2}{\bz_{n}}{}{}{})$ & $\infty$ &
        40 & $(\amal{D_2}{D_{2n}}{D_{2}}{\bz_1}{\tD_1})$ & $4$ &
        74 & $(\amal{\bz_3}{SO(2)}{}{}{})$ & $8$ \\
      7 & $(\amal{D_3}{\bz_{n}}{}{}{})$ & $\infty$ &
        41 & $(\amal{D_6}{D_{2n}}{D_{2}}{\bz_3}{D_3})$ & $4$ &
        75 & $(\amal{D_2}{SO(2)}{}{}{})$ & $2$ \\
      8 & $(\amal{\bz_6}{\bz_{n}}{}{}{})$ & $\infty$ &
        42 & $(\amal{D_6}{D_{2n}}{D_{2}}{\bz_3}{\bz_6})$ & $4$ &
        76 & $(\amal{D_3}{SO(2)}{}{}{})$ & $4$ \\
      9 & $(\amal{\tD_3}{\bz_{n}}{}{}{})$ & $\infty$ &
        43 & $(\amal{D_6}{D_{2n}}{D_{2}}{\bz_3}{\tD_3})$ & $4$ &
        77 & $(\amal{\bz_6}{SO(2)}{}{}{})$ & $4$ \\
      10 & $(\amal{D_6}{\bz_{n}}{}{}{})$ & $\infty$ &
        44 & $(\amal{D_3}{D_{3n}}{D_{3}}{\bz_1}{})$ & $4$ &
        78 & $(\amal{\tD_3}{SO(2)}{}{}{})$ & $4$ \\
      11 & $(\amal{D_1}{\bz_{2n}}{\bz_{2}}{\bz_1}{})$ & $\infty$ &
        45 & $(\amal{\tD_3}{D_{3n}}{D_{3}}{\bz_1}{})$ & $4$ &
        79 & $(\amal{D_6}{SO(2)}{}{}{})$ & $2$ \\
      12 & $(\amal{\bz_2}{\bz_{2n}}{\bz_{2}}{\bz_1}{})$ & $\infty$ &
        46 & $(\amal{D_6}{D_{3n}}{D_{3}}{\bz_2}{})$ & $2$ &
        80 & $(\amal{D_1}{O(2)}{D_{1}}{\bz_1}{})$ & $4$ \\
      13 & $(\amal{\tD_1}{\bz_{2n}}{\bz_{2}}{\bz_1}{})$ & $\infty$ &
        47 & $(\amal{D_6}{D_{6n}}{D_{6}}{\bz_1}{})$ & $2$ &
        81 & $(\amal{\bz_2}{O(2)}{D_{1}}{\bz_1}{})$ & $12$ \\
      14 & $(\amal{D_2}{\bz_{2n}}{\bz_{2}}{D_1}{})$ & $\infty$ &
        48 & $(\amal{\bz_1}{D_{n}}{}{}{})$ & $24$ &
        82 & $(\amal{\tD_1}{O(2)}{D_{1}}{\bz_1}{})$ & $4$ \\
      15 & $(\amal{D_2}{\bz_{2n}}{\bz_{2}}{\bz_2}{})$ & $\infty$ &
        49 & $(\amal{D_1}{D_{n}}{}{}{})$ & $4$ &
        83 & $(\amal{D_2}{O(2)}{D_{1}}{D_1}{})$ & $2$ \\
      16 & $(\amal{D_2}{\bz_{2n}}{\bz_{2}}{\tD_1}{})$ & $\infty$ &
        50 & $(\amal{\bz_2}{D_{n}}{}{}{})$ & $12$ &
        84 & $(\amal{D_2}{O(2)}{D_{1}}{\bz_2}{})$ & $2$ \\
      17 & $(\amal{D_3}{\bz_{2n}}{\bz_{2}}{\bz_3}{})$ & $\infty$ &
        51 & $(\amal{\tD_1}{D_{n}}{}{}{})$ & $4$ &
        85 & $(\amal{D_2}{O(2)}{D_{1}}{\tD_1}{})$ & $2$ \\
      18 & $(\amal{\bz_6}{\bz_{2n}}{\bz_{2}}{\bz_3}{})$ & $\infty$ &
        52 & $(\amal{\bz_3}{D_{n}}{}{}{})$ & $8$ &
        86 & $(\amal{D_3}{O(2)}{D_{1}}{\bz_3}{})$ & $4$ \\
      19 & $(\amal{\tD_3}{\bz_{2n}}{\bz_{2}}{\bz_3}{})$ & $\infty$ &
        53 & $(\amal{D_2}{D_{n}}{}{}{})$ & $2$ &
        87 & $(\amal{\bz_6}{O(2)}{D_{1}}{\bz_3}{})$ & $4$ \\
      20 & $(\amal{D_6}{\bz_{2n}}{\bz_{2}}{D_3}{})$ & $\infty$ &
        54 & $(\amal{D_3}{D_{n}}{}{}{})$ & $4$ &
        88 & $(\amal{\tD_3}{O(2)}{D_{1}}{\bz_3}{})$ & $4$ \\
      21 & $(\amal{D_6}{\bz_{2n}}{\bz_{2}}{\bz_6}{})$ & $\infty$ &
        55 & $(\amal{\bz_6}{D_{n}}{}{}{})$ & $4$ &
        89 & $(\amal{D_6}{O(2)}{D_{1}}{D_3}{})$ & $2$ \\
      22 & $(\amal{D_6}{\bz_{2n}}{\bz_{2}}{\tD_3}{})$ & $\infty$ &
        56 & $(\amal{\tD_3}{D_{n}}{}{}{})$ & $4$ &
        90 & $(\amal{D_6}{O(2)}{D_{1}}{\bz_6}{})$ & $2$ \\
      23 & $(\amal{\bz_3}{\bz_{3n}}{\bz_{3}}{\bz_1}{})$ & $\infty$ &
        57 & $(\amal{D_6}{D_{n}}{}{}{})$ & $2$ &
        91 & $(\amal{D_6}{O(2)}{D_{1}}{\tD_3}{})$ & $2$ \\
      24 & $(\amal{\bz_6}{\bz_{3n}}{\bz_{3}}{\bz_2}{})$ & $\infty$ &
        58 & $(\amal{D_1}{D_{2n}}{\bz_{2}}{\bz_1}{})$ & $4$ &
        92 & $(\amal{\bz_1}{O(2)}{}{}{})$ & $12$ \\
      25 & $(\amal{\bz_6}{\bz_{6n}}{\bz_{6}}{\bz_1}{})$ & $\infty$ &
        59 & $(\amal{\bz_2}{D_{2n}}{\bz_{2}}{\bz_1}{})$ & $12$ &
        93 & $(\amal{D_1}{O(2)}{}{}{})$ & $2$ \\
      26 & $(\amal{D_1}{D_{n}}{D_{1}}{\bz_1}{})$ & $8$ &
        60 & $(\amal{\tD_1}{D_{2n}}{\bz_{2}}{\bz_1}{})$ & $4$ &
        94 & $(\amal{\bz_2}{O(2)}{}{}{})$ & $6$ \\
      27 & $(\amal{\bz_2}{D_{n}}{D_{1}}{\bz_1}{})$ & $24$ &
        61 & $(\amal{D_2}{D_{2n}}{\bz_{2}}{D_1}{})$ & $2$ &
        95 & $(\amal{\tD_1}{O(2)}{}{}{})$ & $2$ \\
      28 & $(\amal{\tD_1}{D_{n}}{D_{1}}{\bz_1}{})$ & $8$ &
        62 & $(\amal{D_2}{D_{2n}}{\bz_{2}}{\bz_2}{})$ & $2$ &
        96 & $(\amal{\bz_3}{O(2)}{}{}{})$ & $4$ \\
      29 & $(\amal{D_2}{D_{n}}{D_{1}}{D_1}{})$ & $4$ &
        63 & $(\amal{D_2}{D_{2n}}{\bz_{2}}{\tD_1}{})$ & $2$ &
        97 & $(\amal{D_2}{O(2)}{}{}{})$ & $1$ \\
      30 & $(\amal{D_2}{D_{n}}{D_{1}}{\bz_2}{})$ & $4$ &
        64 & $(\amal{D_3}{D_{2n}}{\bz_{2}}{\bz_3}{})$ & $4$ &
        98 & $(\amal{D_3}{O(2)}{}{}{})$ & $2$ \\
      31 & $(\amal{D_2}{D_{n}}{D_{1}}{\tD_1}{})$ & $4$ &
        65 & $(\amal{\bz_6}{D_{2n}}{\bz_{2}}{\bz_3}{})$ & $4$ &
        99 & $(\amal{\bz_6}{O(2)}{}{}{})$ & $2$ \\
      32 & $(\amal{D_3}{D_{n}}{D_{1}}{\bz_3}{})$ & $8$ &
        66 & $(\amal{\tD_3}{D_{2n}}{\bz_{2}}{\bz_3}{})$ & $4$ &
        100 & $(\amal{\tD_3}{O(2)}{}{}{})$ & $2$ \\
      33 & $(\amal{\bz_6}{D_{n}}{D_{1}}{\bz_3}{})$ & $8$ &
        67 & $(\amal{D_6}{D_{2n}}{\bz_{2}}{D_3}{})$ & $2$ &
        101 & $(\amal{D_6}{O(2)}{}{}{})$ & $1$ \\
      34 & $(\amal{\tD_3}{D_{n}}{D_{1}}{\bz_3}{})$ & $8$ &
        68 & $(\amal{D_6}{D_{2n}}{\bz_{2}}{\bz_6}{})$ & $2$ &
      & & \\
    \bottomrule
  \end{tabular}
      \caption{Conjugacy Classes of Subgroups in $D_6\times O(2)$}\label{tab: Tab1-D6}
\end{table}

\begin{remark}\rm
For two groups $\mathcal G$ and $\mathcal  O(2)$, the introduced notation $\mathcal{H}=H^\varphi\times_L^\psi K\subset \mathcal G\times\mathcal O(2)$, where $H\subset \mathcal G$, $K\subset \mathcal O(2) $ are two subgroups and $\varphi:H\rightarrow L$ and $\psi:H\rightarrow L$ are epimorphismns, needs to be simplified.  It was proposed by  Hao-Pin Wu (cf \cite{Pin}) to denote the conjugacy classes $(\mathcal{H})$ in a more comprehensive way.   To be more precise, in order to 
identify $L$ with $K/\text{Ker\,}(\psi)\subset O(2)$ and denote by $r$ the
rotation generator in $L$. Next we put 
\begin{align*}
Z=\text{Ker\,}(\varphi)\quad \text{ and }\quad R=\varphi^{-1}(\left<r\right>)
\end{align*}
and define 
\begin{equation}  \label{eq:amalg}
\mathcal{H}=:H\prescript{Z}{R}\times_{L}K,.
\end{equation}
Of course, in the case when all the epimorphisms $\varphi$ with the kernel $Z $ are conjugate, there is no need to use the symbol $Z$ in \eqref{eq:amalg}
and we will simply write $\mathcal{H}=H\prescript{Z}{}\times_{L}K$.
Moreover, in this case all epimorphisms $\varphi$ from $H$ to $L$ are
conjugate, we can also omit the symbol $L$, i.e. we will write $\mathcal{H}=H\prescript{}{}\times_{L}K$. The conjugacy classes of subgroups in $D_6\times
O(2)$ are listed Table \ref{tab: Tab1-D6}, which were obtained in \cite{Pin} using G.A.P. programming.
\end{remark}
\vskip.2cm

\section{Model for Atomic Interaction} 

Consider $n$ identical particles $u_j$, $j=0,1,2,\dots, n-1$, in the plane $\bc$. Assume that each particle interacts with the adjacent particles  $u_{j-1}$ and $u_{j+1}$, where the indices $j-1$ and $j+1$ are taken mod $n$.  

Put $u:=(u_0,u_1,\dots,u_{n-1})^T\in \bc^n$ and 
$\Omega'_o:=\{ u\in\bc^n:  u_k\not= u_j, \text{ if }k\not=j, \text{ with }\,k,\,j=0,\,1,\,2,\,\cdots,\,n-1\}$. 
\vs
We define the following energy functional $V:\Omega'_o\to \br$ by 
\begin{equation}\label{eq:pot}
\boldsymbol {\mathcal V}(u):= \sum_{j=0}^n U(|u_{j+1}-u_j|^2) + \sum_{0\le j<k\le n-1} W(|u_j-u_k|^2),
\end{equation}
where 
\[
U(t)=t-2\sqrt t, \quad W(t)=\frac{B}{t^6}-\frac A{t^3}+\frac \sigma{\sqrt t}, \quad t>0.
\]
\vs 
The following Newtonian equation describes the interaction between these $n$-particles,
\begin{equation}\label{eq:mol}
\ddot u= -\nabla^2 \bV(u), \quad u\in \Omega'_o.
\end{equation} 
\vs
\subsection{Symmetric Equilibrium for \eqref{eq:mol}}\label{sec:equilib}
 We notice that the space $\bc^n$ is a representation of the group $D_n\times O(2)$, where
$D_n$ stands for the dihedral group corresponding to the group of symmetries of a regular $n$-gone, which can be described as a subgroup of the symmetric group $S_n$ of $n$-elements $\{0,1,\dots,n-2,n-1\}$. More precisely, $D_n$ is  generated by
the ``rotation'' $\xi:= (0,1,2,3,\dots,n-1)$ and the ``reflection'' $\kappa:=(1,n-1)(2,n-2)\dots(m,n-m)$, where $m=\left\lfloor \frac{n-1}2\right\rfloor$. Then the action of $\mathfrak G$ on $\bc^n$ is given by
\begin{equation}\label{eq:act1}
(\sigma, A)(z_0,z_1,\dots,z_{n-1})= (Az_{\sigma(0)},Az_{\sigma(1)},\dots,Az_{\sigma(n-1)}),
\end{equation}
where $A\in O(2)$ and $\sigma\in D_n\subset S_n$ is given as a permutation of the vertices $C_n\subset \bc$ of the $n$-gone. To be more precise we have 
\begin{align*}
(\xi, A)(z_0,z_1,\dots,z_{n-2},z_{n-1})&= (Az_{n-1},Az_{0},\dots,Az_{n-3},Az_{n-2}),\\
(\kappa,A)(z_0,z_1,\dots,z_{n-2},z_{n-1})&=(z_0,z_{n-1},\dots,z_1).
\end{align*}
Let us point out that the group $D_n$ can be also described as a subgroup of $O(2)$,
where $\kappa$ is identified to the complex conjugation $\left[ \begin{array}{cc}1&0\\0&-1  \end{array}  \right]$ and $\xi^j$ is identified to the complex number  $\gamma^j$, $\gamma:=e^{i\frac{2\pi}{n}}$, representing  the rotation 
\[
\gamma^j=\left[ \begin{array}{cc}\cos\frac{2\pi j}{n}&-\sin\frac{2\pi j}{n}\\\sin\frac{2\pi j}{n}&\cos\frac{2\pi j}{n}  \end{array}  \right].
\]
It is evident that $\Omega'_o\subset \bc^n$ is $G$-invariant. 
Notice that the function $\bV:\Omega'_o\to \br$ is invariant with respect to the action of $\bc$ on $\bc^n$ by shifting, that is, for all $z\in \bc$ and  $(z_0,z_1,\dots,z_{n-1})\in \Omega'_o$ we have
\[
 \bV(z_0+z,z_1+z,\dots,z_{n-1}+z)=\bV(z_0,z_1,\dots,z_n).
\]
Therefore in order to make System~(\ref{eq:mol}) reference point independent, we put 
\begin{equation}\label{eq:V}
\mathscr V:=\{(z_0,z_1,z_2,\dots,z_{n-1})\in \bc^n: z_0+z_1+z_2+\dots+z_{n-1}=0\},
\end{equation}
and $\Omega_o=\Omega'_o\cap \mathscr V$.  Then, we obtain that $\mathscr V$ and $\Omega_o$ are $G$-invariant and in addition $\Omega_o$ is flow-invariant for \eqref{eq:mol}. 
\vs
Consider the point 
$v^o:=(1,\gamma,\gamma^2,\dots,\gamma^{n-1})\in \Omega_o$, where $\gamma:=e^{i\frac{2\pi}{n}}$.
One can verify that  the isotropy group $\mathfrak G_{v^o}$, which for simplicity we denote by $\Gamma$, is given as the following amalgamated subgroup of $D_n\times O(2)$
\[
\Gamma:=D_n\times_{D_n}D_n:=\{(g,g)\in D_n\times D_n: g\in D_n\},
\]
where in order to consider $D_n$ as a subgroup of $O(2)$.
Then $\mathscr V^\Gamma$ is a one dimensional subspace of $\mathscr V$ and we have (see \eqref{eq:V0} for more details) that 
\[
\mathscr V^\Gamma= \text{span}_{\br}\{(1,\gamma,\gamma^{2},\dots,\gamma^{n-1})\}.
\]
Then, by  Symmetric Criticality Condition, a critical point of $\bV^\Gamma:\Omega_o^\Gamma\to \br$ is also a critical point of $\bV$. 
Notice that $V$ satisfies the   {\it coercivity} condition, that is, $V(u)\to \infty$ as $u$ approaches $\partial \Omega_o$ or $\|u\|\to\infty$. Then  there exists a global minimum point $u^o$ of $\bV^\Gamma$ in $\Omega_o^\Gamma$, which implies that $\nabla \bV^\Gamma(u^o)=0$.  Since $\mathscr V^\Gamma$ is two-dimensional, we denote 
its vectors by  $v:=z(1,\gamma,\gamma^{2},\dots,\gamma^{n-1})$, $z\in \bc$. One can look for the orbit of equilibria for \eqref{eq:mol} where $z=t$, $t\in \br$.  
Put
\begin{equation}\label{eq:const}
a=4\sin^2\frac{\pi}{n}\quad \text { and }\quad a_{jk}=4\sin^2\frac{(k-j)\pi}{n},
\end{equation}
where $0\le j<k\le n-1$.
Define  
\begin{align}\label{phi-def}
\phi(t)= n U\left(at^2\right) + \sum_{0\le j<k\le n-1} W\left(a_{jk}t^2\right), \quad t>0.
\end{align}
Notice that $|\gamma^{j+1}-
\gamma^j|=\left|2\sin\frac{\pi}{n}\right|$ and $|\gamma^{k}-
\gamma^j|=2\left|\sin\frac{(k-j)\pi}{n}\right|$, $\phi$ is exactly the restriction of $\bV$ to the fixed-point subspace $\mathscr V_o^\Gamma=\{t(1,\gamma,\gamma^{2},\dots,\gamma^{n-1}): t>0\}$, thus in order to find an equilibrium for \eqref{eq:mol}, by Symmetric Criticality Principle, it is sufficient to identify a critical point $r_o$ of $\phi(t)$. Clearly, 
\[
\lim_{t\to 0^+}\phi(t)=\lim_{t\to\infty} \phi(t)=\infty,
\] 
thus there exists a minimizer $r_o\in(0,\infty)$, which is a critical point of $\vp$ and consequently 
\begin{equation}\label{eq:crit-u0}
u^o=r_o(1,\gamma,\gamma^{2},\dots,\gamma^{n-1})\in \Omega_o
\end{equation}
is the $\Gamma$-symmetric equilibrium of $\bV$, providing   the configuration of particles, shown at Figure \ref{fig:1}  being a stationary solution to \eqref{eq:mol}. 
\begin{figure}
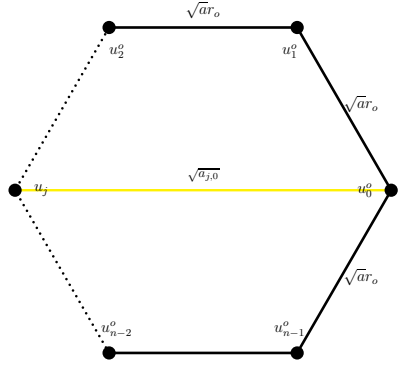

\vglue2.5cm\hskip.5cm
\scalebox{.5}{
\psline[linecolor=yellow,linewidth=1.8pt](-5,0)(5,0)
\rput(0,.4){$\sqrt{a_{j,0}}$}
\rput(4.3,0){\Large $u_0^o$}
\rput(2.3,3.7){\Large $u_1^o$}
\rput(-2.3,3.7){\Large $u_2^o$}
\rput(-2.3,-3.7){\Large $u_{n-2}^o$}
\rput(2.3,-3.7){\Large $u_{n-1}^o$}
\rput(-4.3,0){\Large $u_j$}
\rput(0,4.8){\Large $\sqrt ar_o$}
\rput(4.2,2.3){\Large $\sqrt ar_o$}
\rput(4.2,-2.3){\Large $\sqrt ar_o$}
\psline[linestyle=dotted,linewidth=2pt](-2.500000000, 4.330127020)(-5., 0.)(-2.500000000, -4.330127020)
\psline[linewidth=2pt](-2.500000000, -4.330127020)(2.500000000, -4.330127020)
\psline[linewidth=2pt](5,0)(2.500000000, 4.330127020)(-2.500000000, 4.330127020)
\psline[linewidth=2pt](2.500000000, -4.330127020)(5,0)
\psdot[dotsize=10pt](5,0)
\psdot[dotsize=10pt](2.500000000, 4.330127020)
\psdot[dotsize=10pt](-2.500000000, 4.330127020)
\psdot[dotsize=10pt](-5., 0.)
\psdot[dotsize=10pt](-2.500000000, -4.330127020)
\psdot[dotsize=10pt](2.500000000, -4.330127020)
}
\vskip2.2cm
\caption{Stationary solution to equation \eqref{eq:mol1}  with dihedral symmetries.}\label{fig:1}
\end{figure}
In the following, we write $u^o=(u_0^o,u_1^o,\dots,u_{n-1}^o)\in \bc^n$, where $u_k^o=r_o\gamma^k$, $k=0,1,2,\dots, n-1$ with $\gamma=e^{i\frac{2\pi}{n}}$. 
\vs
\subsection{$G$-Equivariant Bifurcation Problem for \eqref{eq:mol}}
In what follows, we are interested in finding non-trivial $p$-periodic solutions to  \eqref{eq:mol}, bifurcating from the orbit equilibrium points $\mathfrak {G}(u^o)$. By normalizing the period, namely, by making the substitution $v(t)=u\left(\frac{pt}{2\pi}  \right)$ in \eqref{eq:mol} we obtain the following system, 
\begin{equation}\label{eq:mol1}
\begin{cases}
\ddot v=-\lambda^2 \nabla \bV(v),\\
v(0)=v(2\pi),\;\; \dot v(0)=\dot v(2\pi),
\end{cases}
\end{equation} 
where $\lambda=\frac p{2\pi}$.
 Since system~\eqref{eq:mol1} is symmetric with respect to the group action $G=D_n\times O(2)$, we have the orbit of equilibria $M=G(u^o)$, which is a one-dimensional submanifold in $\mathscr V$, with the tangent vector at $u^o\in M$ being 
\[
v^o=i(1,\gamma,\dots,\gamma^{n-1}).
\] 
Notice that the slice $S_o$  to the orbit $M$ at $u^o$ is $S_o=\{ z\in \bc^n: z\bullet v^o=0\}$, which implies that 
\[
S_o=\left\{(z_0,z_1,\dots,z_{n-1})\in \bc^n: \re\left(\sum_{k=0}^{n-1} z_k i\gamma^{-k}  \right)=0    \right\}.
\]
Notice that the $G$-isotropy group of $u_o$  is given by,
\[
\Gamma=G_{u_o}= D_n\times_{D_n}D_n=\{(g,g)\in D_n\times D_n: g\in D_n\},
\]
which  can be  identified with $D_n$. 
\vs

\subsection{$\Gamma$-Isotypical Decomposition of $S_o$}
Let  $\mu:\bc^n\to \br$ be given by
\[
\mu(z_0,z_1,\dots,z_{n-1})= \re\left(\sum_{k=0}^{n-1} z_k i\gamma^{-k}  \right).
\]
Then we have $\ker \mu=S_o$. Denote also by $\nu :\bc^n\to \bc$ the complex linear functional 
\[
\nu(z_0,z_1,\dots,z_{n-1})= \sum_{k=0}^{n-1} z_k i\gamma^{-k}. 
\]
Then we also have $\ker \nu\subset \ker \mu$.
Notice that for $j=0,1,\dots,n-2$, and $z_k=\gamma^{-jk}$, $k=0,1,2,\dots,n-1$, we have 
\begin{align*}
\sum_{k=0}^{n-1} z_k i\gamma^{-k}&
=i\sum_{k=0}^{n-1} \gamma^{-(j+1)k}
=i\frac{1-\gamma^{-(j+1)n}}{1-\gamma^{-(j+1)}}=
0,
\end{align*}
which implies for every $z\in\bc$,
\[
  (z,z\gamma^{-j},z\gamma^{-2j},\dots,z\gamma^{-(n-1)j})\in \ker \nu.
\]
Put
\[
W_j=\left\{ (z,z\gamma^{-j},z\gamma^{-2j},\dots,z\gamma^{-(n-1)j}):\; z\in \bc\right\}\subset \bc^n.
\]
The action of $\Gamma$ on $W_j$  which we identify with $D_n$ can be  described as follows. Put  \[
\bold z=(z,z\gamma^{-j},z\gamma^{-2j},\dots,z\gamma^{-(n-1)j}).
\] Then we have
\begin{align*}
(\gamma,\gamma)\bold z&=(\gamma,\gamma)(z,z\gamma^{-j},z\gamma^{-2j},\dots,z\gamma^{-(n-1)j})\\
&=(\gamma^{-(n-1)j+1}z,\gamma z,\gamma^{-j+1}z,\gamma^{-2j+1}z,\dots,\gamma^{-(n-2)j+1}z)\\
&=(\gamma^{j+1}z,\gamma z,\gamma^{-j+1}z,\gamma^{-2j+1}z,\dots,\gamma^{-(n-2)j+1}z)\\
&=\gamma^{j+1}\cdot (z,z\gamma^{-j},z\gamma^{-2j},\dots,z\gamma^{-(n-1)j})\\
& =\gamma^{j+1}\cdot \bold z,
\end{align*}
where `$\cdot$' denotes the usual complex multiplication.
Therefore,  the sub-representation $W_{j}$ is equivalent to the irreducible $D_n$-representation $\mathcal V_{j+1}$ (see \cite{AED} for more details).
On the other hand, if $j=n-1$, then   for  $z_k=z\gamma^{-(n-1)k}$, $k=0,1,2,\dots,n-1$, $z\in \bc$, we have 
\begin{align*}
\sum_{k=0}^{n-1} z_k i\gamma^{-k}&=\sum_{k=0}^{n-1} z\gamma^{-j(n-1)k} i\gamma^{-k}=niz,
\end{align*}
which implies that $(z_0,z_1,\dots,z_{n-1})\in \ker \mu$ if and only if $z$ is real. Therefore, we have  
\begin{align*}
 W_{n-1}&= \{(x,x\gamma^{-(n-1)}, \dots, x\gamma^{-1}): x\in \br\}\\
 &=\{(x,x\gamma, \dots, x\gamma^{(n-1)}): x\in \br\}\\
 &=\text{span\,}(v_o)\\
 & =\mathcal V_0,
 \end{align*}
which is an additional component of $S_o$ on which $(\gamma,\gamma)$ and $(\kappa,\kappa)$ act trivially, namely,  it is equivalent to $\mathcal V_0$. 
\vs
Comparing the dimension of the slice with the dimensions of the irreducible components of $S_o$, we recognize that this decomposition is the complete decomposition of $S_o$ into a product of irreducible $D_n$-sub-representations and, as a consequence, we are now able to identify the $D_n$-isotypical components of $S_o$. 
\vs
\paragraph{The Case of $n$ being an odd number:} In this case,
the space $S_o$ has the following isotypical components,
\[
S_o=\mathscr V_0\oplus \mathscr V_1\oplus \mathscr V_2\oplus\dots\oplus \mathscr V_j\oplus \dots \oplus \mathscr V_r, \text{ with }\,r=\left\lfloor \frac{n}2\right\rfloor,
\]
where   $\mathscr V_0=W_{n-1}$, that is,
\begin{align}
\mathscr V_0&=\text{span}_{\br}\{(1,\gamma^{-(n-1)},\gamma^{-2(n-1)},\dots,\gamma^{-(n-1)(n-1)})\}\notag\\
&=\text{span}_{\br}\{(1,\gamma,\gamma^{2},\dots,\gamma^{n-1})\},\label{eq:V0}
\end{align}
$\mathscr V_1=W_{n-2}$ and 
for $\left\lfloor \frac{n}2\right\rfloor\ge j>1$,
\[
\mathscr V_j=W_{j-1}\oplus W_{n-j-1}.
\]
Put 
\begin{align*}
\bold u^j&=(1,\gamma^{-(j-1)},\gamma^{-2(j-1)},\dots,\gamma^{-(n-1)(j-1)}),\\
\bold v^j&=(1,\gamma^{(j+1)},\gamma^{2(j+1)},\dots,\gamma^{(n-1)(j+1)}).
\end{align*}
Then for $j=1,2,\dots, r$,  $\mathscr V_j$ is a complex subspace of $\bc^n$ such that
\[
\mathscr V_j= \text{span}_{\bc}\{\bold u^j\}\oplus \text{span}_{\bc}\{\bold v^j\}.
\]
\paragraph{The Case of $n$ being an even number:} In this case,
the space $S_o$ has the following isotypical components,
\[
S_o=\mathscr V_0\oplus\mathscr V_1\oplus  \mathscr V_2\oplus\dots\oplus \mathscr V_j\oplus \dots \oplus \mathscr V_r, \text{ with }\, r=\frac{n}2, 
\]
with the similar components $\mathscr V_j$ for $j<\frac n2=r$, and an additional isotypical component $\mathscr V_r$, given by
\begin{align*}
\mathscr V_r&:=W_{r-1}=  \text{span}_{\bc}\{(1,\gamma^{-(r-1)},\gamma^{-2(r-1)},\dots,\gamma^{-(n-1)(r-1)})\}\\
&=\text{span}_{\bc}\{(1,-\gamma,\gamma^2,\dots,(-1)^k\gamma^k,\dots,-\gamma^{n-1})\}.
\end{align*}
\vs
\section{Variational Reformulation of \eqref{eq:mol1}}

\paragraph{\bf Sobolev Space of $V$-Valued Periodic Functions:}
Since $\mathscr V$ is an orthogonal  $\Gamma$- representation, we can consider the  the first Sobolev
space of $2\pi$-periodic functions from $\mathbb{R}$ to $\mathscr V$, that is, 
\begin{equation*}
H^1_{2\pi}(\mathbb{R},\mathscr V)=\{z:\mathbb{R}\to \mathscr V\;:\; z(0)=z(2\pi), \;
z|_{[0,2\pi]}\in H^1([0,2\pi];\mathscr V)\},
\end{equation*}
equipped with the inner product 
\begin{equation*}
\langle z,w\rangle=\int_0^{2\pi}(\dot z(t)\bullet \dot w(t)+z(t)\bullet
w(t))dt,\quad z,w\in H^1_{2\pi}(\mathbb{R},\mathscr V).
\end{equation*}
Let $O(2)$ denote the group of $2\times 2$-orthogonal matrices.
Notice that $O(2)=SO(2)\cup SO(2)\kappa$, where $\kappa=\left[
\begin{array}{cc}
1 & 0 \\ 
0 & -1
\end{array}
\right]$. It is convenient to identify a rotation 
\[
\left[
\begin{array}{cc}
\cos \tau & -\sin\tau \\ 
\sin \tau & \cos \tau
\end{array}
\right]\in SO(2),
\] with $e^{i\tau}\in S^1\subset \mathbb{C}$. Notice that $\kappa e^{i\tau}=e^{-i\tau}\kappa$, that is, $\kappa$ as a linear transformation of $\bc$ into itself, acts as complex conjugation. 
The space $H^1_{2\pi}(\mathbb{R},\mathscr V)$
is an orthogonal Hilbert representation of $\mathfrak G= (D_n\times O(2))\times O(2)$. Indeed,
we have for $z\in H^1_{2\pi}(\mathbb{R},\mathscr V)$, $(\xi,A)\in D_n\times O(2)$ and  $e^{i\tau}\in S^1$ we
have (see \eqref{eq:act1})
\begin{align}
\Big(\big(\xi , A \big),e^{i\tau}\Big)x)(t) &= (\xi,A) x(t+\tau),\label{eq:ac1} \\
\Big(\big(\xi,A\big) ,e^{i\tau}\kappa)x\Big)(t) &=( \xi,A) x(-t+\tau).\label{eq:ac2}
\end{align}
We identify a $2\pi$-periodic function $x:\mathbb{R}\to V$
with a function $\widetilde x:S^1\to \mathscr V$ via the following commuting diagram:
\newline
\vskip2cm\hskip4cm 
\rput(0,2){$\mathbb{R}$}
\rput(3,2){$S^1$}
\psline{->}(0.5,2)(2.5,2) 
\rput(1.5,0){$\mathscr V$}
\psline{->}(0.2,1.7)(1.3,.3)
\psline{->}(2.8,1.7)(1.8,.3)
\rput(1.5,2.3){$\mathfrak e$}
\rput(1,1.){$x$} 
\rput(2,1){$\widetilde x $}
\rput(5,1){$\mathfrak e(\tau)=e^{i\tau}$} 

\vskip.3cm
 \noindent
Using this identification, we  write $H^1(S^1,\mathscr V)$ in place of $H^1_{2\pi}(\mathbb{R},\mathscr V)$.
Let 
\[
\Omega=\{u\in H^1(S^1,\mathscr V): u(t)\in \Omega_o, \text{ for all } {t\in \br}\}.
\] 
Then system~\eqref{eq:mol1} can be written as the following variational equation
\begin{equation}\label{eq:bif1}
\nabla_u J(\lambda,u)=0, \quad (\lambda,u)\in \br\times \Omega,
\end{equation}
where $J:\br\times \Omega \to \br$ is defined by
\begin{equation}\label{eq:var-1}
J(\lambda,u)=\int_0^{2\pi} \left[ \frac 12 |\dot u(t)|^2-\lambda^2 \bV(u(t))  \right]dt.
\end{equation}
Assume that $u^o=r_o(1,\gamma,\gamma^2,\gamma^{n-1})\in \bc^n$ is the equilibrium point of \eqref{eq:mol} described in subsection \ref{sec:equilib}. Then $u^o$ is a critical point of $J$. We are interested in finding non-stationary $2\pi$-periodic solutions bifurcating from $u^o$, that is, non-constant solutions to system \eqref{eq:bif1}. Notice that $\mathfrak G_{u_o}=D_n\times O(2)$,  where $D_n\simeq D_n\times_{D_n}D_n\subset D_n\times O(2)$. We consider the $\mathfrak G$-orbit $\mathfrak G(u^o)$ in the space $H^1(S^1,\mathscr V)$. We denote by $\mathscr H$ the slice to $\mathfrak G(u^o)$ in $H^1(S^1,\mathscr V)$. We will also denote by $\mathscr J:\br\times \wt\Omega \to \br$ the restriction of $J$ to the set $\br\times \wt \Omega$, where $\wt \Omega=\mathscr H\cap \Omega$. Put $\bold G=\mathfrak G_{u^o}$. Then $\mathscr J$ is $\bold G$-invariant. Then by the Slice Criticality Principle (see Theorem \ref{thm:SCP}), critical points of $\mathscr J$ are critical points of $J$ and  are solutions to  system \eqref{eq:bif1}.
\vs
Consider the operator $L: H^2(S^1;\mathscr V)\to L^2(S^1;\mathscr V)$, given by $Lu=-\ddot u+u$, $u\in H^2(S^1,\mathscr V)$. Then the inverse operator $L^{-1}$ exists and is bounded.
Let $j: H^2(S^1;\mathscr V)\to H^1(S^1,\mathscr V)$ be the natural embedding operator. Then $j$ is a compact operator and we have
\begin{equation}\label{eq:gradJ}
\nabla _uJ(\lambda,u)=u-j\circ L^{-1} (\lambda^2 N_{\nabla \bV}(u)+u),\quad u\in H^1(S^1,\mathscr V),
\end{equation}
where 
\begin{align*}
N_{\nabla \bV}(u)(t)=\nabla \bV (u(t)), \quad t\in \br.
\end{align*}
Consequently, the bifurcation problem \eqref{eq:bif1} can be written as 
\[
u-j\circ L^{-1} (\lambda^2 N_{\nabla \bV}(u)+u)=0.
\]
Moreover, we have
\begin{equation}\label{eq:D2J}
\nabla ^2_uJ(\lambda,u^o)=\id-j\circ L^{-1} (\lambda^2 N_{\nabla^2 \bV}(u^o)+\Id),\quad u\in H^1(S^1,\mathscr V),
\end{equation}
where 
\begin{align*}
(N_{\nabla^2 \bV}(u^o)v)(t)=\nabla^2 \bV (u^o)v(t), \quad v\in H^1(S^1,\mathscr V); \; t\in \br.
\end{align*}
Consider the operator  $\mathscr A(\lambda): \mathscr H\to \mathscr H$, given by 
\begin{equation}\label{eq:opA}
\mathscr A(\lambda):=P(\nabla ^2_uJ(\lambda,u^o)). 
\end{equation}
Notice that 
\[
\nabla_u^2 \mathscr J(\lambda, u^o)=\mathscr A(\lambda),
\]
thus, by implicit function theorem,  $\mathfrak G(u^o)$ is an isolated orbit of critical points of $J$, whenever $\mathscr A(\lambda)$ is an isomorphism. Therefore, a point $(\lambda_o, u^o)$ is a bifurcation point for \eqref{eq:bif1}, then $\mathscr A(\lambda_o)$ is not an isomorphism. In such a case we put $\Lambda=\{\lambda>0: \mathscr A(\lambda_o)$ is not an isomorphism$\}$, and call the set $\Lambda$ the {\it critical set} for the trivial solution $u^o$. 
\vs

\subsection{Application of Equivariant Gradient Degree}
Consider the $S^1$-action on $H^1(S^1,\mathscr V)$, where $S^1$ acts on functions by shifting the argument (see  \eqref{eq:ac1}). Then, $(H^1(S^1,\mathscr V))^{S^1}$ is the space of constant functions, which can be identified with the space $\mathscr V$, that is, 
\[
H^1(S^1,\mathscr V)=\mathscr V\oplus \mathscr W, \quad \mathscr W:=\mathscr V^\perp.
\]
Then, the slice $\mathscr S_o$ in $H^1(S^1,\mathscr V)$ to the orbit $\mathfrak G(u^o)$ at $u^o$ is exactly
\[
\mathscr S_o=S_o\oplus \mathscr W.
\]
\begin{definition}\rm
 We say that $\lambda_o\in \Lambda$ satisfies {\it  condition} (C) if $P_o\circ \mathscr A(\lambda_o)|_{S_o}:S_o\to S_o$, where $P_o:\mathscr V\to S_o$ is an orthogonal projection, is an isomorphism.
\end{definition}
\vs 
\begin{theorem}\label{th:bif1}
Consider the bifurcation system \eqref{eq:bif1} and assume that $\lambda_o\in \Lambda$ satisfies condition (C) and is isolated in the critical critical set $\Lambda$, i.e. there exists $\lambda_-<\lambda_o<\lambda_+$ such that $[\lambda_-,\lambda_+]\cap \Lambda=\{\lambda_o\}$. Define 
\begin{equation}\label{eq:eq-inv}
\omega(\lambda_o):=\Theta\left[ \nabla_{\bold G}\text{\rm -deg} \Big(\mathscr A(\lambda_-),B_1(0)\Big)- \nabla_{\bold G}\text{\rm -deg} \Big(\mathscr A(\lambda_+),B_1(0)\Big)  \right],
\end{equation}
where $B_1(0)$ stands for the open unit ball in $\mathscr H$. If 
\[\omega(\lambda_o)=n_1(H_1)+n_2(H_2)+\dots +n_m(H_m),\]
is non-zero, i.e. $n_j\not=0$ for some $j=1,2,\dots,m$, then there exists a bifurcating branch of nontrivial solutions to \eqref{eq:bif1} from the orbit $\{\lambda_o\}\times \mathfrak G(u^o)$ with symmetries at least $(H_j)$.
\end{theorem}
Consider the $S^1$-isotypical decomposition of $\mathscr W$, that is, 
\[
\mathscr W=\overline{\bigoplus_{l=1}^\infty \mathscr W_l}, \quad \mathscr W_l:=\{\cos(l\cdot) \mathfrak a+\sin(l\cdot)\mathfrak b: \mathfrak a, \, \mathfrak b\in V\}.
\]
In a standard way, the space $\mathscr W_l$, $l=1,2,\dots$, can be naturally identified with the space $\mathscr V^c$ on which $S^1$ acts by $l$-folding. To be more precise, 
\[
\mathscr W_l=\{e^{il\cdot}z: z\in \mathscr V\}.
\]
Notice that, since the operator $\mathscr A(\lambda)$ is $ \bold G$-equivariant where $\bold G=D_n\times O(2)$, it is also $S^1$-equivariant  and thus $\mathscr A(\lambda)(\mathscr W_l)\subset \mathscr W_l$. On the other hand,  we have 
\begin{align}\label{eq:spect-l}
\sigma(\mathscr A(\lambda)|_{\mathscr W_l})=\left\{1-\frac{\lambda^2 \mu +1}{l^2+1} : \mu\in \sigma(\nabla_u^2 \bV(u^o))  \right\}, 
\end{align}
which under condition (C) implies that $\lambda_o\in \Lambda$ if and only if $\lambda_o^2=\frac{l^2}\mu$ for some $l=1,\,2,\,3,\dots$ and $ \mu\in \sigma(\nabla_u^2 \bV(u^o))$.

\vs
\section{Computation of the Spectrum $\sigma(\nabla^2   {\mathcal V}(u^o))$ }
\vs
\paragraph{Computation of $\nabla \bV(u)$:}
Since the potential $\bV$ is given by \eqref{eq:pot}, we can write that $\bV(u)=\Phi(u)+\Psi(u)$, $u\in \Omega$, where 
\[
\Phi(u)=\sum_{j=0}^n U(|u_{j+1}-u_j|^2) , \quad \Psi(u)= \sum_{0\le j<k\le n-1} W(|u_j-u_k|^2).
\]
Notice that 
\[
\nabla\Phi(u)=2\left[ \begin{array}{c} U'(|u_{0}-u_{n-1}|^2)(u_{0}-u_{n-1})+ U'(|u_{0}-u_{1}|^2)(u_{0}-u_{1})\\ 
U'(|u_{1}-u_{2}|^2)(u_{1}-u_{2})+ U'(|u_{1}-u_{0}|^2)(u_{1}-u_{0})\\ 
\vdots\\
U'(|u_{n-1}-u_{0}|^2)(u_{n-1}-u_{0})+ U'(|u_{n-1}-u_{n-2}|^2)(u_{n-1}-u_{n-2})
\end {array}  \right], 
\]
and 
\[
\nabla\Psi(u)=2\left[ \begin{array}{c} \sum_{k\not=0} W'(|u_0-u_k|^2)(u_0-u_k)\\
 \sum_{k\not=1} W'(|u_1-u_k|^2)(u_1-u_k)\\
 \vdots\\
  \sum_{k\not=n-1} W'(|u_{n-1}-u_k|^2)(u_{n-1}-u_k).
  \end {array}  \right]
\]

\vs

\paragraph{Computation of $\nabla^ 2\bV(u^o)$:}
For a given complex number $z=x+iy$, which we write in a vector form $z=(x,y)^T$, we define the matrix $\mathfrak m_z:=zz^T$, i.e.
\[
\mathfrak m_z:=\left[\begin{array} {c}x\\y \end{array}  \right][x,y]=\left[\begin{array}{cc} x^2 &xy\\xy&y^2 \end{array}  \right].
\]
We will also apply the following notation
\[
z_{jk}=\gamma^{j}-\gamma^{k}, \quad \gamma=e^{i\frac{2\pi}{n}},\;\; j,k\in \bz,
\]
and we put $\mathfrak m_{jk}:=\mathfrak m_{z_{jk}}$. Noticing that
\begin{align*}
\re(\gamma^j-\gamma^k)& =2\sin\frac{(k-j)\pi}{n}\sin\frac{(k+j)\pi}{n},
\end{align*} we know that the $2\times 2$ matrix $\mathfrak m_{jk}$ can be described using the complex operators as 
\[
\mathfrak m_{jk}=\frac{|z_{jk}|^2}{2}\Big[1-\gamma^{j+k}\kappa  \Big]=2\sin^2\frac{\pi(j-k)}{n}\Big[1-\gamma^{j+k}\kappa  \Big].
\]
Put 
\begin{align*}
v_{jk}&:= U'(a_{jk}r_o^2), &v_{jj}&:=0\\
u_{jk}&=2U''(a_{jk}r_o^2)\sin^2\frac{\pi (j-k)}{n}, &u_{jj}&:=0\\
\mathfrak v_{jk}&:= W'(a_{jk}r_o^2), &\mathfrak v_{jj}&:=0\\
\mathfrak u_{jk}&=2W''(a_{jk}r_o^2)\sin^2\frac{\pi (j-k)}{n}, &\mathfrak u_{jj}&:=0\\
\end{align*}
Notice that we have 
\[
v_{j+l,k+l}=v_{jk}, \quad u_{j+l,k+l}=u_{jk}.
\]
\vs
By direct computations we have that 
\[
\nabla^2\Phi(u^o)= 2U'(ar^2_o)\mathcal A+4r_o^2U''(ar^2_o)\mathcal B,
\]
where 
\begin{equation*}\label{eq:Phi-A}\
\mathcal A:=\left[ \begin{array} {ccccc}2&-1&0&\dots &-1\\
-1&2&-1&\dots&0\\
\vdots&\vdots&\vdots&\ddots &\vdots\\
 -1&0&0&\dots&2 \end{array}    \right],
\end{equation*}
and 
{\small \begin{align*}\label{eq:Phi-B}\
\mathcal B&:=\left[ \begin{array} {ccccc}\mathfrak m_{0,n-1}+\mathfrak m_{0,1}&-\mathfrak m_{01}&0&\dots &-\mathfrak m_{0,n-1}\\
-\mathfrak m_{1,0}&\mathfrak m_{1,0}+\mathfrak m_{1,2}&-\mathfrak m_{1,2}&\dots&0\\
\vdots&\vdots&\vdots&\ddots &\vdots\\
 -\mathfrak m_{n-1,0}&0&0&\dots&\mathfrak m_{n-1,n-2}+\mathfrak m_{n-1,0} \end{array}    \right]\\
 &=2\sin^2\frac \pi n\left[ \begin{array} {ccccc} 2-(\gamma^{-1}+\gamma)\kappa & -1+\gamma\kappa&0&\dots&-1+\gamma^{-1}\kappa \\
 -1+\gamma\kappa&2-(\gamma^{1}+\gamma^3)\kappa &-1+\gamma^{3}\kappa&\dots &0 \\
 \vdots&\vdots&\vdots&\ddots &\vdots\\
 -1+\gamma^{-1}\kappa &0&0&\dots &2-(\gamma^{-3}+\gamma^{-1})
\end{array}\right]\\
&=2\sin^2\frac \pi n\Big[ \mathcal A-\mathcal B_\gamma K \Big],
\end{align*}}
where 
\[   
K =\left[ \begin{array} {ccccc}\kappa&0&0&\dots &0\\
0&\kappa&0&\dots&0\\
\vdots&\vdots&\vdots&\ddots &\vdots\\
 0&0&0&\dots&\kappa \end{array}    \right],\; \mathcal B_\gamma=\left[ \begin{array} {ccccc}\gamma^{-1}+\gamma&-\gamma&0&\dots &-\gamma^{-1}\\
-\gamma&\gamma+\gamma^3&-\gamma^3&\dots&0\\
\vdots&\vdots&\vdots&\ddots &\vdots\\
 -\gamma^{-1}&0&0&\dots&\gamma^{-3}+\gamma^{-1}\end{array}    \right]
\]

Next, by direct computations one can derive the following matrix form of $\nabla ^2\Psi(u^o)=2\mathcal C+4r_o^2\mathcal D$, where
\begin{align*}\label{eq:Psi-C}
\mathcal C&:=
\left[ \begin{array} {cccc}\ds\sum_{j\not=0}W'(a_{0j}r_o^2)&-W'(a_{01}r_o^2)&\dots &-W'(a_{0,n-1}r_o^2) \\
-W'(a_{1,0}r_o^2)&\ds\sum_{j\not=1} W'(a_{1j}r_o^2)&\dots&-W'(a_{2,n-1}r_o^2)\\
\vdots&\vdots&\ddots &\vdots\\
-W'(a_{n-1,0}r_o^2)&-W'(a_{n-1,1}r_o^2)&\dots&\ds\sum_{j\not=n-1} W'(a_{n-1,j}r_o^2) \end{array}    \right]\\
&=\left[ \begin{array} {ccccc}\ds\sum_j \mathfrak v_{0j}&-\mathfrak v_{01}&-v_{02}&\dots&-\mathfrak v_{0,n-1}\\
-\mathfrak v_{10}&\ds\sum_j \mathfrak v_{1j}&-\mathfrak v_{12}&\dots&-\mathfrak v_{1,n-1}\\
\vdots&\vdots&\vdots&\ddots &\vdots\\
-\mathfrak v_{n-1,0}&-\mathfrak v_{n-1,1}&-\mathfrak v_{n-1,2}&\dots&\ds\sum_j \mathfrak v_{n-1,j}
\end{array}    \right]
\end{align*}
and 
{\scriptsize 
\begin{align*}
\mathcal  D&:=
\left[ \begin{array} {cccc}\ds\sum_{j\not=0}W''(a_{0j}r_o^2)\mathfrak m_{0j}&-W''(a_{01}r_o^2)\mathfrak m_{01}&\dots &-W''(a_{0,n-1}r_o^2) \mathfrak m_{0,n-1}\\
-W''(a_{1,0}r_o^2)\mathfrak m_{1,0}&\ds\sum_{j\not=1} W''(a_{1j}r_o^2)\mathfrak m_{1,j}&\dots&-W''(a_{1,n-1}r_o^2)\mathfrak m_{1,n-1}\\
\vdots&\vdots&\ddots &\vdots\\
-W''(a_{n-1,0}r_o^2)\mathfrak m_{n-1,0}&-W''(a_{n-1,1}r_o^2)\mathfrak m_{n-1,1}&\dots&\ds\sum_{j\not=n-1} W''(a_{n-1,j}r_o^2)\mathfrak m_{n-1,j} \end{array}    \right]\\
&=\left[ \begin{array} {cccc}
\ds\sum_j \mathfrak u_{0j}(1-\gamma^j\kappa) &-\mathfrak u_{01}(1-\gamma \kappa)& \dots &-\mathfrak u_{0,n-1}(1-\gamma^{n-1}\kappa)\\
-\mathfrak u_{10}(1-\gamma\kappa)&\ds\sum_j \mathfrak u_{1j}(1-\gamma^{j+1}\kappa) & \dots &-\mathfrak u_{1,n-1}(1-\kappa)\\
\vdots&\vdots&\ddots &\vdots\\
-\mathfrak u_{n-1,0}(1-\gamma^{-1}\kappa)& -\mathfrak u_{n-1,1}(1-\kappa)&\dots&   \ds\sum_j \mathfrak u_{n-1,j}(1-\gamma^{j-1}\kappa)
\end{array}    \right]\\
&=\wt {\mathcal C} - \mathcal D_\gamma K,
\end{align*}}
where
\[
\wt{\mathcal C}=\left[ \begin{array} {ccccc}\ds\sum_j \mathfrak u_{0j}&-\mathfrak u_{01}&-\mathfrak u_{02}&\dots&-\mathfrak u_{0,n-1}\\
-\mathfrak u_{10}&\ds\sum_j \mathfrak u_{1j}&-\mathfrak u_{12}&\dots&-\mathfrak u_{1,n-1}\\
\vdots&\vdots&\vdots&\ddots &\vdots\\
-\mathfrak u_{n-1,0}&-\mathfrak u_{n-1,1}&-\mathfrak u_{n-1,2}&\dots&\ds\sum_j \mathfrak u_{n-1,j}
\end{array}    \right]
\]
and 
\[
\mathcal D_\gamma :=\left[ \begin{array} {cccc}\ds\sum_j \mathfrak u_{0j}\gamma^j&-\mathfrak u_{01}\gamma&\dots&-\mathfrak u_{0,n-1}\gamma^{-1}\\
-\mathfrak u_{10}\gamma&\ds\sum_j \mathfrak u_{1j}\gamma^{j+1}&\dots&-\mathfrak u_{1,n-1}\\
\vdots&\vdots&\ddots &\vdots\\
-\mathfrak u_{n-1,0}\gamma^{-1}&-\mathfrak u_{n-1,1}&\dots&\ds\sum_j \mathfrak u_{n-1,j}\gamma^{j-1}
\end{array}    \right].
\]
\vs
\noi{\bf Case 1:  $n$ being an odd number:}  Put $\mathcal A_k:=\mathcal A|_{\mathscr V_k}$, then we have the following matrices:
\begin{align*}
\mathcal A_0&=[2-2\re(\gamma)]=\Big[ 4\sin^2\frac \pi n  \Big],\;\;\;\; \mathcal A_1=[2-2\re(\gamma^2)]=\Big[ 4\sin^2\frac {2\pi} n  \Big]\\
 \mathcal A_k&=\left[\begin{array} {cc} 2-2\re(\gamma^{k-1})&0 \\
0&2-2\re(\gamma^{k+1})  \end{array}   \right]=\left[\begin{array} {cc}4\sin^2\frac{\pi(k-1)}{n}&0 \\
0&4\sin^2\frac{\pi(k+1)}{n}  \end{array}   \right],
\end{align*}
where $ 1<k<\left\lfloor \frac{n}2\right\rfloor$.
Next, put $\mathcal B_k:=\mathcal B|_{\mathscr V_k}$.  Notice that 
\begin{align*}
\mathcal B_\gamma K \bold u^k&=2\re ( \gamma-\gamma ^k) \bold v^k=-4\sin\frac{\pi(1+k)}{n}\sin\frac{\pi(1-k)}{n}\bold v^k,\\
\mathcal B_\gamma K \bold v^k&=2\re ( \gamma-\gamma ^k) \bold u^k=-4\sin\frac{\pi(1+k)}{n}\sin\frac{\pi(1-k)}{n}\bold u^k.
\end{align*}
Then, by direct computations we get
\begin{gather*}
\mathcal B_0=\Big[8\sin^4\frac \pi n\Big],\;\;\mathcal B_1=\Big[8\sin^2\frac \pi n\sin^2\frac{2\pi}n\Big]\\
 \mathcal B_k=\left[\begin{array} {cc}8\sin^2\frac \pi n\sin^2\frac{\pi(k-1)}{n}&8\sin^2\frac \pi n \sin\frac{\pi(1+k)}{n}\sin\frac{\pi(1-k)}{n}\\
8\sin^2\frac \pi n\sin\frac{\pi(1+k)}{n}\sin\frac{\pi(1-k)}{n}&8\sin^2\frac \pi n\sin^2\frac{\pi(k+1)}{n} \end{array}   \right],
\end{gather*}
where
$ 1<k<\left\lfloor \frac{n}2\right\rfloor$.
Notice that 
\[
\mathcal C \bold u^k= \left( \sum_j \mathfrak v_{0j}-\sum_ j \mathfrak v_{0j}\gamma^{(1-k)j} \right)\bold u^k,\;\;\; \mathcal C \bold v^k= \left( \sum_j \mathfrak v_{0j}-\sum_ j \mathfrak v_{0j}\gamma^{(1+k)j} \right)\bold v^k
\]
In addition (since $n$ is odd), we put
\begin{align*}
\bold a_k&:=\sum_j \mathfrak v_{0j}-\sum_ j \mathfrak v_{0j}\gamma^{(1-k)j} =\sum_{j=1}^{\lfloor\frac{n-1}2\rfloor} 4\mathfrak v_{0j}\sin^2\frac{\pi(1-k)j}{n}\\
\bold b_k&:=\sum_j \mathfrak v_{0j}-\sum_ j \mathfrak v_{0j}\gamma^{(1+k)j} =\sum_{j=1}^{\lfloor\frac{n-1}2\rfloor} 4\mathfrak v_{0j}\sin^2\frac{\pi(1+k)j}{n}\\
\mathfrak a_k&:=\sum_j \mathfrak u_{0j}-\sum_ j \mathfrak u_{0j}\gamma^{(1-k)j} =\sum_{j=1}^{\lfloor\frac{n-1}2\rfloor} 4\mathfrak u_{0j}\sin^2\frac{\pi(1-k)j}{n}\\
\mathfrak b_k&:=\sum_j \mathfrak u_{0j}-\sum_ j \mathfrak u_{0j}\gamma^{(1+k)j} =\sum_{j=1}^{\lfloor\frac{n-1}2\rfloor} 4\mathfrak u_{0j}\sin^2\frac{\pi(1+k)j}{n}
\end{align*} 
Then, put
 $\mathcal C_k:=\mathcal C|_{\mathscr V_k}$, $\wt {\mathcal C}_k:=\wt{\mathcal C}|_{\mathscr V_k}$ and
  \begin{alignat*}{3}
\mathcal C_0&=\left[\bold a_0     \right],\;\mathcal C_1=0, \;\;&\wt{ \mathcal C}_0=\left[\mathfrak a_0     \right],\;\;\wt{\mathcal C}_1=0\\
 \mathcal C_k&=\left[\begin{array} {cc} \bold a_k&0\\
 0&\bold b_k \end{array}   \right],\;\;  &\wt{\mathcal C}_k=\left[\begin{array} {cc} \mathfrak a_k&0\\
 0&\mathfrak  b_k \end{array}   \right]
\end{alignat*}
where
$ 1<k<\left\lfloor \frac{n}2\right\rfloor$. Next, notice that 
\begin{align*}
\mathcal D_\gamma K\bold u^k&=\left( \sum_j \mathfrak u_{0j}\gamma^j-\sum_j \mathfrak u_{0j}\gamma^j\gamma^{j(k-1)}\right)\bold v^k\\
\mathcal D_\gamma K\bold v^k&=\left( \sum_j \mathfrak u_{0j}\gamma^j-\sum_j \mathfrak u_{0j}\gamma^j\gamma^{j(k-1)}\right)\bold u^k.
\end{align*}
Put 
\begin{align*}
\boldsymbol {\mathfrak c}_k&=  \sum_j \mathfrak u_{0j}\gamma^j-\sum_j \mathfrak u_{0j}\gamma^j\gamma^{j(k-1)}=4\sum_{j=1}^{\lfloor \frac{n-1}2 \rfloor} \mathfrak u_{0j}\sin\frac{\pi j(k+1)}{n}\sin\frac{\pi j(k-1)}{n}.
\end{align*}
Then
\begin{gather*}
\mathcal D_0=[\mathfrak a_0]   ,\;\;\mathcal D_1=\left[  \bold c_1  \right]=[0],\\
 \mathcal D_k=\left[\begin{array} {cc} \ds  \mathfrak a_k & -\boldsymbol {\mathfrak c}_k\\
 -\boldsymbol {\mathfrak c}_k &\mathfrak b_k\end{array}   \right],\quad \text{ for }\; 1<k<\lfloor \frac n2\rfloor
\end{gather*}
Put 
\[
\mathscr L:=\nabla^2\bV(u^o),\;\; \mathscr L_k:= \nabla^2\bV(u^o)|_{\mathscr V_k}, \quad 0\le k\le \left\lfloor\frac n2\right\rfloor
\]
Then we have
\begin{equation}\label{eq:eig0}
\mathscr L_0=[\boldsymbol{\alpha}_0  ]
\end{equation}
where
 \[
\boldsymbol{\alpha}_0=8\left( U'(ar_o^2) +4r_o^2 U''(ar_o^2)\sin^2\frac \pi n\right) \sin^2\frac \pi n+\sum_{j=1}^{\lfloor\frac{n-1}2\rfloor} (8v_{0j}+16r_o^2u_{0j})\sin^2\frac{\pi j}{n}\]
also
\begin{equation}\label{eq:eig1}
\mathscr A_1=[\boldsymbol\alpha_1  ]
\end{equation}
where
\[
\boldsymbol\alpha_1=8\left(U'(ar_o^2)+4r_o^2 U''(ar_o^2)\sin^2\frac {2\pi} n\right)\sin^2\frac{2\pi}n
\]
and finally
\begin{equation}\label{eq:eigk}
\mathscr L_k=\left[\begin{array}{cc} 2\boldsymbol\alpha_k &\boldsymbol \delta_k\\
\boldsymbol\delta_k&2\boldsymbol\beta_k \end{array}   \right], \quad 1<k\le \left\lfloor \frac {n-1}2 \right\rfloor
\end{equation}
where
\begin{align*}
\boldsymbol\alpha_k&:=4U'(ar_o^2)\sin^2\frac{\pi(k-1)}{n}+16r_o^2U''(ar_o)\sin^2\frac{\pi}{n}\sin^2\frac{\pi(1-k)}{n}\\
&+\sum_{j=1}^{\lfloor \frac{n-1}2 \rfloor} \left(4\mathfrak v_{0j}+8r_o^2\mathfrak u_{oj}\right)\sin^2\frac{\pi(1-k)j}{n}\\
\boldsymbol\beta_k&:=4U'(ar_o^2)\sin^2\frac{\pi(k+1)}{n}+16r_o^2U''(ar_o)\sin^2\frac{\pi}{n}\sin^2\frac{\pi(1+k)}{n}\\
&+\sum_{j=1}^{\lfloor \frac{n-1}2 \rfloor} \left(4\mathfrak v_{0j}+8r_o^2\mathfrak u_{oj}\right)\sin^2\frac{\pi(1+k)}{n}\\
\boldsymbol\delta_k&:= 32r_o^2 U''(ar_o^2) \sin^2\frac \pi n \sin\frac{\pi(1+k)}{n}\sin\frac{\pi(1-k)}{n}\\
&-32r_o^2 \sum_{j=1}^{\lfloor \frac{n-1}{2}\rfloor} \mathfrak u_{0j} \sin\frac{\pi j(k-1)}{n}\sin \frac{\pi j(k+1)}{n}.
\end{align*}
Then we have the following explicit formulae for the spectrum
\[
\sigma\left( \mathscr A)   \right)=\left\{ \mu_0,\mu_1,\mu^{\pm}_k, \;\;  1<k\le \left\lfloor \frac {n-1}2 \right\rfloor \right\},
\]
where
\begin{align*}
\mu_0&:=\boldsymbol \alpha_0,\\
\mu_1&:=\boldsymbol \alpha_1,\\
\mu_k^\pm&:=\boldsymbol \alpha_k+\boldsymbol \beta_k\pm \sqrt{(\boldsymbol \alpha_k-\boldsymbol \beta_k)^2+\boldsymbol \delta_k^2}, \;\;\;1<k \le \left\lfloor \frac {n-1}2 \right\rfloor.
\end{align*}
Notice that, for each eigenvalue $\mu\in \sigma(\mathscr A)$ we have that its ${\cV_k}$-isotypical multiplicity $m_k(\mu)$ is given by
\begin{equation}\label{eq:k-mult}
m_k(\mu)=\begin{cases} 
1 &\text{ if } \mu=\mu_k\text{ or } \mu=\mu_k^+\text{ or }\mu=\mu_k^-,\\
0 &\text{ otherwise}.
\end{cases}
\end{equation} 
\vs
\noi{\bf Case 2:  $n$ being an even number:}  In this case we have an additional isotypical component $\mathscr V_r$, $r=\frac n2$, and the entries of the matrix  $\mathscr A_k$ and are slightly different. More precisely, notice that in this case we have 
\begin{align*}
\bold a_k&:=\sum_j \mathfrak v_{0j}-\sum_ j \mathfrak v_{0j}\gamma^{(1-k)j} =\sum_{j=1}^{\lfloor\frac{n-1}2\rfloor} 4\mathfrak v_{0j}\sin^2\frac{\pi(1-k)j}{n}+2\mathfrak v_{0r}\delta(k) \\
\bold b_k&:=\sum_j \mathfrak v_{0j}-\sum_ j \mathfrak v_{0j}\gamma^{(1+k)j} =\sum_{j=1}^{\lfloor\frac{n-1}2\rfloor} 4\mathfrak v_{0j}\sin^2\frac{\pi(1+k)j}{n}+2\mathfrak v_{0r}\delta(k) \\
\mathfrak a_k&:=\sum_j \mathfrak u_{0j}-\sum_ j \mathfrak u_{0j}\gamma^{(1-k)j} =\sum_{j=1}^{\lfloor\frac{n-1}2\rfloor} 4\mathfrak u_{0j}\sin^2\frac{\pi(1-k)j}{n}+2\mathfrak u_{0r}\delta(k)\\\
\mathfrak b_k&:=\sum_j \mathfrak u_{0j}-\sum_ j \mathfrak u_{0j}\gamma^{(1+k)j} =\sum_{j=1}^{\lfloor\frac{n-1}2\rfloor} 4\mathfrak u_{0j}\sin^2\frac{\pi(1+k)j}{n}+2\mathfrak u_{0r}\delta(k)\\
\boldsymbol {\mathfrak c}_k&=  \sum_j \mathfrak u_{0j}\gamma^j-\sum_j \mathfrak u_{0j}\gamma^j\gamma^{j(k-1)}=4\sum_{j=1}^{\lfloor \frac{n-1}2 \rfloor} \mathfrak u_{0j}\sin\frac{\pi j(k+1)}{n}\sin\frac{\pi j(k-1)}{n}\\
&-2\delta(k)\mathfrak u_{0r},
\end{align*} 
where $\delta(k):=\begin{cases}1 &\text{ if } k \text { is even}\\
0& \text { if } k \text{ is odd}  \end{cases}$. Therefore, 
we have 
\begin{gather*}
\mathscr L_0=\left[\boldsymbol\alpha_0  \right],\;\;\mathscr L_1=\left[\boldsymbol\alpha_1    \right],\;\; 
\mathscr L_r=\left[\boldsymbol \alpha_r  \right]\\
\mathscr L_k=\left[\begin{array}{cc} 2\boldsymbol\alpha_k &\boldsymbol \delta_k\\
\boldsymbol\delta_k&2\boldsymbol\beta_k \end{array}   \right], \quad 1<k\le \left\lfloor \frac {n-1}2 \right\rfloor
\end{gather*}
where 
\begin{align*}
\boldsymbol\alpha_0&=8\left( U'(ar_o^2) +4r_o^2 U''(ar_o^2)\sin^2\frac \pi n\right) \sin^2\frac \pi n+\sum_{j=1}^{\lfloor\frac{n-1}2\rfloor} (8v_{0j}+16r_o^2u_{0j})\sin^2\frac{\pi j}{n}\\
&+4v_{0r}+8u_{0r} r_o^2 ,\\
\boldsymbol\alpha_1&=8\left(U'(ar_o^2)+4r_o^2 U''(ar_o^2)\sin^2\frac {2\pi} n\right)\sin^2\frac{2\pi}n,\\
\boldsymbol\alpha_r&:=8U'(ar_o^2)\cos^2\frac \pi n+32r_o^2 U''(ar_o^2)\sin^2 \frac \pi n\cos^2\frac \pi n \\
&+\sum_{j=1}^{\lfloor \frac{n-1}{2}\rfloor} (8\mathfrak v_{0j}+16r_o^2\mathfrak  u_{0j})\sin ^2\frac{\pi(1-r)j}{n}+(4\mathfrak v_{0r}+8 r_o^2\mathfrak u_{0r}),
\end{align*}
and for $1<k\le r-1$
\begin{align*}
\boldsymbol\alpha_k&:=4U'(ar_o^2)\sin^2\frac{\pi(k-1)}{n}+16r_o^2U''(ar_o)\sin^2\frac{\pi}{n}\sin^2\frac{\pi(1-k)}{n}\\
&+\sum_{j=1}^{\lfloor \frac{n-1}2 \rfloor} \left(4\mathfrak v_{0j}+8r_o^2\mathfrak u_{oj}\right)\sin^2\frac{\pi(1-k)j}{n}+\delta(k)(2\mathfrak v_{0r}+4r_o^2 \mathfrak u_{0r}),\\
\boldsymbol\beta_k&:=4U'(ar_o^2)\sin^2\frac{\pi(k+1)}{n}+16r_o^2U''(ar_o)\sin^2\frac{\pi}{n}\sin^2\frac{\pi(1+k)}{n}\\
&+\sum_{j=1}^{\lfloor \frac{n-1}2 \rfloor} \left(4\mathfrak v_{0j}+8r_o^2\mathfrak u_{oj}\right)\sin^2\frac{\pi(1+k)}{n}+\delta(k)(2\mathfrak v_{0r}+4r_o^2 \mathfrak u_{0r}),\\
\boldsymbol\delta_k&:=  32r_o^2 U''(ar_o^2) \sin^2\frac \pi n \sin\frac{\pi(1+k)}{n}\sin\frac{\pi(1-k)}{n}\\
&-32r_o^2 \sum_{j=1}^{\lfloor \frac{n-1}{2}\rfloor} \mathfrak u_{0j} \sin\frac{\pi j(k-1)}{n}\sin \frac{\pi j(k+1)}{n}+16\delta(k)r_o^2\mathfrak  u_{0r}.
\end{align*}
Consequently, 
we have the following explicit formulae for the spectrum
\[
\sigma\left( \mathscr A)   \right)=\left\{ \mu_0,\,\mu_1,\,\mu_r,\,\mu^{\pm}_k, \;\;  1<k\le \left\lfloor \frac {n-1}2 \right\rfloor \right\},
\]
where
\begin{align*}
\mu_0&:=\boldsymbol \alpha_0,\\
\mu_1&:=\boldsymbol \alpha_1,\\
\mu_r&:=\boldsymbol \alpha_r,\\
\mu_k^\pm&:=\boldsymbol \alpha_k+\boldsymbol \beta_k\pm \sqrt{(\boldsymbol \alpha_k-\boldsymbol \beta_k)^2+\boldsymbol \delta_k^2}, \;\;\;1<k \le \left\lfloor \frac {n-1}2 \right\rfloor.
\end{align*}
Of course, in this case the formula \eqref{eq:k-mult} is still valid.
\vs

\section{Formulation of Results}
\subsection{Computation of the Gradient $\bold G$-Equivariant Degree}
In order to describe the $\bold G$-isotypical decomposition of the slice $\mathcal S_o$, first, we identify the irreducible $\bold G$-representations related to the isotypical decomposition of $\mathscr W$. These representations are $\mathcal W_{j,l}:=\mathcal V_j\otimes \mathcal U_j$, where $\mathcal U_l$ is the $l$-th irreducible $O(2)$-representation (listed according to the convention introduced in \cite{AED}), $j=0,1,\dots, \lfloor\frac n2\rfloor$, $l=1,2,3,\dots$. The corresponding to $\mathcal W_{jl}$ isotypical components of $\mathscr W$ are 
\[
\mathscr W_{jl}:=\{\cos(l \cdot)\mathfrak a +\sin(l\cdot)\mathfrak b: \mathfrak a,\, \mathfrak b \in \mathscr V_j \}.
\]
These irreducible $\bold G$-representations can be easily described. The representation $\mathcal W_{jl}=\bc\oplus \bc $ is a $4$-dimensional (real)  representation of real type with the action of $\bold G=D_n\times O(2)$ given by the formulae
\begin{align*}
\gamma(z_1,z_2)&:= (\gamma^j\cdot z_1,\gamma^{-j}\cdot z_2), \\
\kappa(z_1,z_2)&:=(z_2,z_1),\\
\xi(z_1,z_1)&:=(\xi^l \cdot z_1,\xi^l\cdot z_2),\\
\bm {\kappa}(z_1,z_2)&:=(\overline z_1,\overline z_2),
\end{align*}
where $\xi \in SO(2)$, $O(2)=SO(2)\cup SO(2)\bm{\kappa}$, and 
\[D_n:=\{1,\gamma,\dots, \gamma^{n-1},\kappa,\gamma\kappa,\dots,\gamma^{n-1}\kappa\}.\]
\vs
For each positive eigenvalue $\mu_j^\pm \in \sigma(\nabla ^2V(ar_o^2))$, $1<j<\lfloor \frac{n-1}2\rfloor$, we define the number $\lambda_{j,l}^\pm:=\frac{l^2}{\mu_j^\pm}$, and for other eigenvalues $\mu_j$, we put $\lambda_{j,l}:=\frac{l^2}{\mu_j}$, $l\in \mathbb N$. Then the critical set $\Lambda$ is composed of exactly all these numbers $\lambda_{j,l}^\pm$ and $\lambda_{j,l}$. Since each of the eigenvalues $\mu_j^\pm$ (for  $1<j<\lfloor \frac{n-1}2\rfloor$) and $\mu_j$ (otherwise) is of $\mathcal V_j$-isotypical multiplicity one, it follows that for 
$\lambda_-<\lambda_o:=\lambda_{j,l}^\pm  <\lambda_+$ (respectively $\lambda_-<\lambda_o:=\lambda_{j,l}  <\lambda_+$), where $[\lambda_-,\lambda_+]\cap \Lambda=\{\lambda_o\}$, thus 
$\sigma_-(\mathscr A(\lambda_-))=\sigma_-(\mathscr A(\lambda_-))\cup \{\lambda_o\}$.
\vs
Consequently, we obtain that for 
\[
\omega_{\bold G}(\lambda_o):=\nabla_{\bold G}\text{\rm -deg} \Big(\mathscr A(\lambda_-),B_1(0)\Big)- \nabla_{\bold G}\text{\rm -deg} \Big(\mathscr A(\lambda_+),B_1(0)\Big) 
\]
is given by
\begin{equation}\label{eq:deg-A}
\omega_{\bold G}(\lambda_o)=\prod_{\xi\in \sigma_-(\mathscr A(\lambda_-))} \prod_{i,k}(\text{\rm Deg}_{\mathcal W_{ik}})^{m_{ik}(\xi)}*\Big(\text{\rm Deg}_{\mathcal W_{jl}}-(\bold G)  \Big).
\end{equation}

\vs
\begin{example}\rm
In the case of the group $\boldsymbol G=D_6\times O(2)$, we have the following basic degrees\footnote{Let us point out that for practical applications of the gradient $D_n\times O(2)$-degree, it is fully justified (see \cite{DKY}) to use is values truncated to the Burnside ring $A(D_n\times O(2))$} $\text{Deg}_{\mathcal W_{jk}}$ (which were obtain in \cite{Pin} using GAP programming) 
\begin{align*}   
 \text{Deg}_{\mathcal W_{0,l}}=\;&(\amal{D_6}{O(2)}{}{}{})-(\amal{D_6}{D_{l}}{}{}{}),\\
\text{Deg}_{\mathcal W_{1,l}}=\;&(\amal{D_6}{O(2)}{}{}{})-(\amal{D_6}{D_{6l}}{D_{6}}{\bz_1}{})-(\amal{D_2}{D_{2l}}{\bz_{2}}{D_1}{})-(\amal{D_2}{D_{2l}}{\bz_{2}}{\tD_1}{})+\\
        &2(\amal{D_2}{D_{2l}}{D_{2}}{\bz_1}{\bz_2})+(\amal{\bz_2}{D_{2l}}{\bz_{2}}{\bz_1}{}),\\
\text{Deg}_{\mathcal W_{2,l}}=\;&(\amal{D_6}{O(2)}{}{}{})-(\amal{D_6}{D_{3l}}{D_{3}}{\bz_2}{})-(\amal{D_2}{D_{2l}}{\bz_{2}}{\bz_2}{})-(\amal{D_2}{D_{l}}{}{}{})+\\
        &2(\amal{D_2}{D_{l}}{D_{1}}{\bz_2}{})+(\amal{\bz_2}{D_{l}}{}{}{}),\\
\text{Deg}_{\mathcal W_{3,l}}=\;&(\amal{D_6}{O(2)}{}{}{})-(\amal{D_6}{D_{2l}}{\bz_{2}}{\tD_3}{}).
\end{align*}
\end{example}
\vs
\subsection{Existence Result} 

\begin{theorem} Under the assumptions formulated in section 3, for every $\lambda_{jl}$, $0<j\le \lfloor \frac n2\rfloor$, there exists an orbit of  bifurcating branches of nontrivial periodic solutions to  \eqref{eq:bif1} from the orbit $\{\lambda_{jl}\}\times \mathfrak G(u^o)$. More precisely, for every  orbit type $ (H_{j,l} )$ in $D_{jl}$ there exists an orbit of periodic solutions with symmetries at least $H_{j,l}$.
\end{theorem}
\begin{proof}This result is a direct consequence of the Existence Property ($\nabla$1) of the gradient equivariant degree formulated in Theorem \ref{thm:Ggrad-properties}.
\end{proof}

\vs
\section{Computational Example}

In this section we consider a dihedral configuration of molecules composed of $n=6$ particles and put  
$A=0.2$, $B=350$, $\sigma=0.25$ for the function $W$ at (\ref{eq:pot}), with which we obtain that $\phi$ defined at (\ref{phi-def}) assumes minimum at $r_0=1.836545792.$ The distinct eigenvalues of the Hessian matrix $\nabla^2 V(u^o)$ are
$\mu_0= -10.36657914,$
 $\mu_1=       43.00585474, $
 $\mu_3 =       19.58406142, $  $\mu_2^-= 7.633501334$ and
 $\mu_2^{+} = 11.42339623$,
Then one can easily compute the critical set $\Lambda$, namely
\begin{align*}
\Lambda=&\Big\{ \lambda_{1,1}=  0.15248819,\;
\lambda_{3,1}=0.22596887 ,\;
\lambda^+_{2,1}= 0.29587099 ,\\
&\hskip.4cm \lambda_{1,2}= 0.30497638,  \;
\lambda^-_{2,1}=0.36194127    ,\;
\lambda_{3,2}= 0.45193775 ,\\
&\hskip.4cm\lambda_{1,3}=0.45746457, \;
\lambda^-_{2,2}= 0.72388254,\;
\lambda_{1,4}=0.60995276,\\
&\hskip.4cm\lambda_{3,3}=0.67790662,\;
\lambda^+_{2,2}=  0.59174197,\;
\lambda_{1,5}=0.76244095,\\
&\hskip.4cm\lambda^+_{2,3}= 0.88761296,\;
\lambda_{3,4}=0.90387549,\;
\lambda_{1,6}=0.91492914,\;\dots 
\Big\}.
\end{align*}
The positive values of the eigenvalues $\xi_{jl}(\lambda)$ of the operator $\mathscr A(\lambda)$  and the critical set $\Lambda$ are illustrated on Figure \ref{fig:eig}.
\vs

\begin{figure}[H]
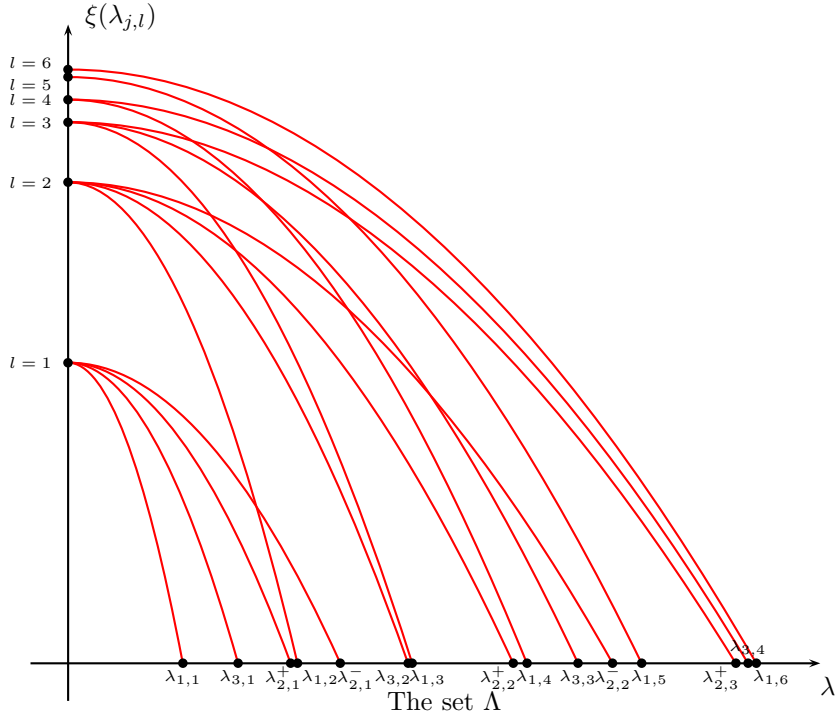
\vskip9cm
\hskip-8cm
\psparabola[linecolor=red](1.5248819,0)(0,4)
\psparabola[linecolor=red](2.2596887,0)(0,4)
\psparabola[linecolor=red](2.9587099,0)(0,4)
\psparabola[linecolor=red](3.0497638,0)(0,6.4)
\psparabola[linecolor=red](3.6194127,0)(0,4)
\psparabola[linecolor=red](4.5193775,0)(0,6.4)
\psparabola[linecolor=red](4.5746457,0)(0,7.2)
\psparabola[linecolor=red](5.9174197,0)(0,6.4)
\psparabola[linecolor=red](6.0995276,0)(0,7.5)
\psparabola[linecolor=red](6.7790662,0)(0,7.2)
\psparabola[linecolor=red](7.2388254,0)(0,6.4)
\psparabola[linecolor=red](7.6244095,0)(0,7.8)
\psparabola[linecolor=red](8.8761296,0)(0,7.2)
\psparabola[linecolor=red](9.0387549,0)(0,7.5)
\psparabola[linecolor=red](9.1492914,0)(0,7.9)
\psline[linecolor=white,fillcolor=white,fillstyle=solid](0,0)(0,9.2)(-10,9.2)(-10,0)
\psline{->}(-.5,0)(10,0)
\rput(10.1,-.3){$\lambda$}
\rput(5,-.5){The set $\Lambda$}
\psline{->}(0,-.5)(0,8.5)
\psdots(0,7.9)(0,7.5)(0,7.8)(0,7.2)(0,6.4)(0,4)
(1.5248819,0)
(2.2596887,0)
(2.9587099,0)
(3.0497638,0)
(3.6194127,0)
(4.5193775,0)
(4.5746457,0)
(5.9174197,0)
(6.0995276,0)
(6.7790662,0)
(7.2388254,0)
(7.6244095,0)
(8.8761296,0)
(9.0387549,0)
(9.1492914,0)
\rput(-.5,4){\tiny $l=1$}
\rput(-.5,6.4){\tiny $l=2$}
\rput(-.5,7.2){\tiny $l=3$}
\rput(-.5,7.5){\tiny $l=4$}
\rput(-.5,7.7){\tiny $l=5$}
\rput(-.5,8){\tiny $l=6$}
\rput(.7,8.6){$\xi(\lambda_{j,l})$}
\rput(1.5248819,-.2){\tiny $\lambda_{1,1}$}
\rput(2.2596887,-.2){\tiny $\lambda_{3,1}$}
\rput(2.8587099,-.2){\tiny $\lambda^+_{2,1}$}
\rput(3.3497638,-.2){\tiny $\lambda_{1,2}$}
\rput(3.8194127,-.2){\tiny $\lambda^-_{2,1}$}
\rput(4.3193775,-.2){\tiny $\lambda_{3,2}$}
\rput(4.7746457,-.2){\tiny $\lambda_{1,3}$}
\rput(5.7174197,-.2){\tiny $\lambda^+_{2,2}$}
\rput(6.1995276,-.2){\tiny $\lambda_{1,4}$}
\rput(6.7790662,-.2){\tiny $\lambda_{3,3}$}
\rput(7.2388254,-.2){\tiny $\lambda^-_{2,2}$}
\rput(7.7244095,-.2){\tiny $\lambda_{1,5}$}
\rput(8.6761296,-.2){\tiny $\lambda^+_{2,3}$}
\rput(9.0387549,.2){\tiny $\lambda_{3,4}$}
\rput(9.3492914,-.2){\tiny $\lambda_{1,6}$}

\vskip.4cm
\caption{Eigenvalues $\xi_{j,l}(\lambda)$ of the operator $\mathscr A(\lambda)$}\label{fig:eig}
\end{figure}

\subsection{Topological Invariants $\omega(\lambda_o)$}
In Table \ref{tab:max} we list the maximal orbit types in $\mathscr W_l\setminus \{0\}$, $l\ge 1$:
\vs
\begin{table}[H]\label{tab:max}
\begin{tabular}{|l|l|}
\hline
$\mathscr W_{jl}$, $l\ge 1$ & maximal orbit types\\
\hline
$\mathscr W_{0l}$ &  $(D_6\times D_l)$\\
$\mathscr W_{1l}$ &$(\amal{D_6}{D_{6l}}{D_{6}}{\bz_1}{})$, $(\amal{D_2}{D_{2l}}{\bz_{2}}{D_1}{})$, $(\amal{D_2}{D_{2l}}{\bz_{2}}{\tD_1}{})$\\
$\mathscr W_{2l}$ &  $(\amal{D_6}{D_{3l}}{D_{3}}{\bz_2}{})$, $(\amal{D_2}{D_{2l}}{\bz_{2}}{\bz_2}{})$, $(\amal{D_2}{D_{l}}{}{}{})$\\
$\mathscr W_{3l}$ &$(\amal{D_6}{D_{2l}}{\bz_{2}}{\tD_3}{})$\\
\hline

\end{tabular}
\caption{Maximal orbit types in $\mathscr W_{jl}$}
\end{table}
\vs
Next, we list the  values of the equivariant invariants $\omega(\lambda_{jl})$ (given by \eqref{eq:eq-inv}):
\begin{align*}
\omega(\lambda_{1,1})&= \text{Deg}_{\mathcal W_{1,1}}-(G)\\
\omega(\lambda_{3,1})&= \text{Deg}_{\mathcal W_{1,1}}*(\text{Deg}_{\mathcal W_{3,1}}-(G))\\
\omega(\lambda^+_{2,1})&= \text{Deg}_{\mathcal W_{1,1}}*\text{Deg}_{\mathcal W_{3,1}}*(\text{Deg}_{\mathcal W_{2,1}}-(G))\\
\omega(\lambda_{1,2})&= \text{Deg}_{\mathcal W_{1,1}}*\text{Deg}_{\mathcal W_{3,1}}*\text{Deg}_{\mathcal W_{2,1}}*(\text{Deg}_{\mathcal W_{1,2}}-(G))\\
\omega(\lambda_{3,1})&= \text{Deg}_{\mathcal W_{1,1}}*\text{Deg}_{\mathcal W_{3,1}}*\text{Deg}_{\mathcal W_{2,1}}*\text{Deg}_{\mathcal W_{1,2}}*(\text{Deg}_{\mathcal W_{2,1}}-(G))\\
\omega(\lambda^+_{2,1})&= \text{Deg}_{\mathcal W_{1,1}}*\text{Deg}_{\mathcal W_{3,1}}*\text{Deg}_{\mathcal W_{2,1}}*\text{Deg}_{\mathcal W_{1,2}}*\text{Deg}_{\mathcal W_{2,1}}*(\text{Deg}_{\mathcal W_{1,2}}-(G))
\end{align*}
These sequence of equivariant invariants $\omega(\lambda_{j,l})$ can be continued indefinitely due to the the fact that any $p$-periodic solution is also $2p$, $3p$, $4p$, etc. periodic solution as well. However, in order to get a clear picture of the emerging from the symmetric equilibrium vibrations, it is sufficient to exhaust all the critical values $\lambda_{j,1}$.  Let us also point out that the exact value of the equivariant invariants $\omega(\lambda_{j,l}$ can be symbolically computed either in its truncated to the  Burnside ring $A(D_n\times O(2))$ (such programs are already available) or in $U(D_n\times O(2))$ (we  have all the needed algorithms so the appropriate computer programs were already  created). However, one should understand that as each equivariant invariant  $\omega(\lambda_{j,l})$ carry the full equivariant topological information about the emerging from the equilibrium $u^o$ periodic vibrations with the limit period $p=2\pi \lambda_{j,l}$, so they can be significantly long. For example, we have 
\begin{align*}
\omega(\lambda^+_{2,1})&= \mathrm{Deg}_{\mathcal{V}_{1,1}}*\mathrm{Deg}_{\mathcal{V}_{3,1}}*
        (\mathrm{Deg}_{\mathcal{V}_{2,1}}-(D_6\times O(2)))\\
        &=-(\amal{D_2}{D_{2}}{\bz_{2}}{\bz_2}{})
        +(\amal{\tD_1}{D_{2}}{\bz_{2}}{\bz_1}{})
        -(\amal{D_2}{D_{1}}{}{}{})\\
        &\hskip.6cm+(\amal{\bz_2}{D_{1}}{}{}{})
        +(\amal{D_1}{D_{1}}{}{}{})
        -(\amal{\bz_1}{D_{1}}{}{}{})\\
        &\hskip.6cm-(\amal{D_6}{D_{3}}{D_{3}}{\bz_2}{})
        +(\amal{\tD_3}{D_{3}}{D_{3}}{\bz_1}{})
        +(\amal{D_3}{D_{3}}{D_{3}}{\bz_1}{})\\
        &\hskip.6cm +(\amal{D_2}{D_{2}}{D_{2}}{\bz_1}{\tD_1})
        +2(\amal{D_2}{D_{1}}{D_{1}}{\bz_2}{})
        +(\amal{D_2}{D_{1}}{D_{1}}{D_1}{})\\
        &\hskip.6cm-2(\amal{\tD_1}{D_{1}}{D_{1}}{\bz_1}{})
        -(\amal{\bz_2}{D_{1}}{D_{1}}{\bz_1}{})
        -2(\amal{D_1}{D_{1}}{D_{1}}{\bz_1}{}).
\end{align*}
\vs
Nevertheless, for the purpose of making predictions about the actual emerging periodic vibration with this particular limit period, one can look in $\omega(\lambda_{j,l})$ for the maximal orbit types listed in Table \ref{tab:max}. 
Therefore, we can list some types\footnote{In order to provide the full list of possible symmetries of the emerfing periodic vibrations one needs to use the full topological invariant $\omega(\lambda_{j,l})$.} (according to their symmetries) of the branches of periodic vibrations  emerging  from  the equilibrium $u^o$:
\begin{itemize}
\item[$\lambda_{1,1}$:] For the limit period $0.95811155$ there exist at least the following three orbits of $p$-periodic vibrations with spatio-temporal symmetries at least $(\amal{D_6}{D_{6}}{D_{6}}{\bz_1}{})$, $(\amal{D_2}{D_{2}}{\bz_{2}}{D_1}{})$,\break $(\amal{D_2}{D_{2}}{\bz_{2}}{\tD_1}{})$.

\item[$\lambda_{3,1}$:] For the limit period $1.41980428$ there exists at least the following orbit of $p$-periodic vibrations with spatio-temporal symmetries at lest $(\amal{D_6}{D_{2}}{\bz_{2}}{\tD_3}{})$. 

\item[$\lambda^+_{2,1}$:] For the limit period $1.85901226$ there exist at least the following three orbits of $p$-periodic vibrations with spatio-temporal symmetries  at least
 $(\amal{D_6}{D_{3l}}{D_{3}}{\bz_2}{})$, $(\amal{D_2}{D_{2l}}{\bz_{2}}{\bz_2}{})$,\break $(\amal{D_2}{D_{l}}{}{}{})$. 

\item[$\lambda_{1,2}$:] For the limit period $1.91622311$ there exist at least the following three orbits of $p$-periodic vibrations with spatio-temporal symmetries  at least $(\amal{D_6}{D_{12}}{D_{6}}{\bz_1}{})$, $(\amal{D_2}{D_{4}}{\bz_{2}}{D_1}{})$,\break $(\amal{D_2}{D_{4}}{\bz_{2}}{\tD_1}{})$.

\item[$\lambda^+_{2,1}$:] For the limit period $2.27414407$ there exist at least the following three orbits of $p$-periodic vibrations with spatio-temporal symmetries  at least
 $(\amal{D_6}{D_{3l}}{D_{3}}{\bz_2}{})$, $(\amal{D_2}{D_{2l}}{\bz_{2}}{\bz_2}{})$,\break $(\amal{D_2}{D_{l}}{}{}{})$
\end{itemize}

Notice that due to the isotypical type of the critical values $\lambda_{j,l}$ there are similar orbit types of branches emerging from $u^o$ with  different limit period. One can ask about the global behavior of such branches. For instance, is it possible that such a branch emerge from one $\lambda_{j,l}$ and then `disappear' into another $\lambda_{j',l'}$? By comparing the values of the equivariant invariants $\omega(\lambda_{j,l})$ with  $\omega(\lambda_{j',l'})$ one can easily say that such situation would be very unlikely possible.

\subsection{Numerical Simulations}
In this subsection, we present some simulations of the periodic solutions predicted by our theoretical. 
 On Figures~\ref{relative-u-43}--\ref{fig-u-11}.
we show the periodic solutions  which were found found for  $\lambda^2=\frac{1}{\mu}$ with $\mu$ taking values near the last four eigenvalues.

\begin{figure}[H]
\begin{center}
\scalebox{0.4}{
\hspace*{-10em}
\includegraphics[angle=0]{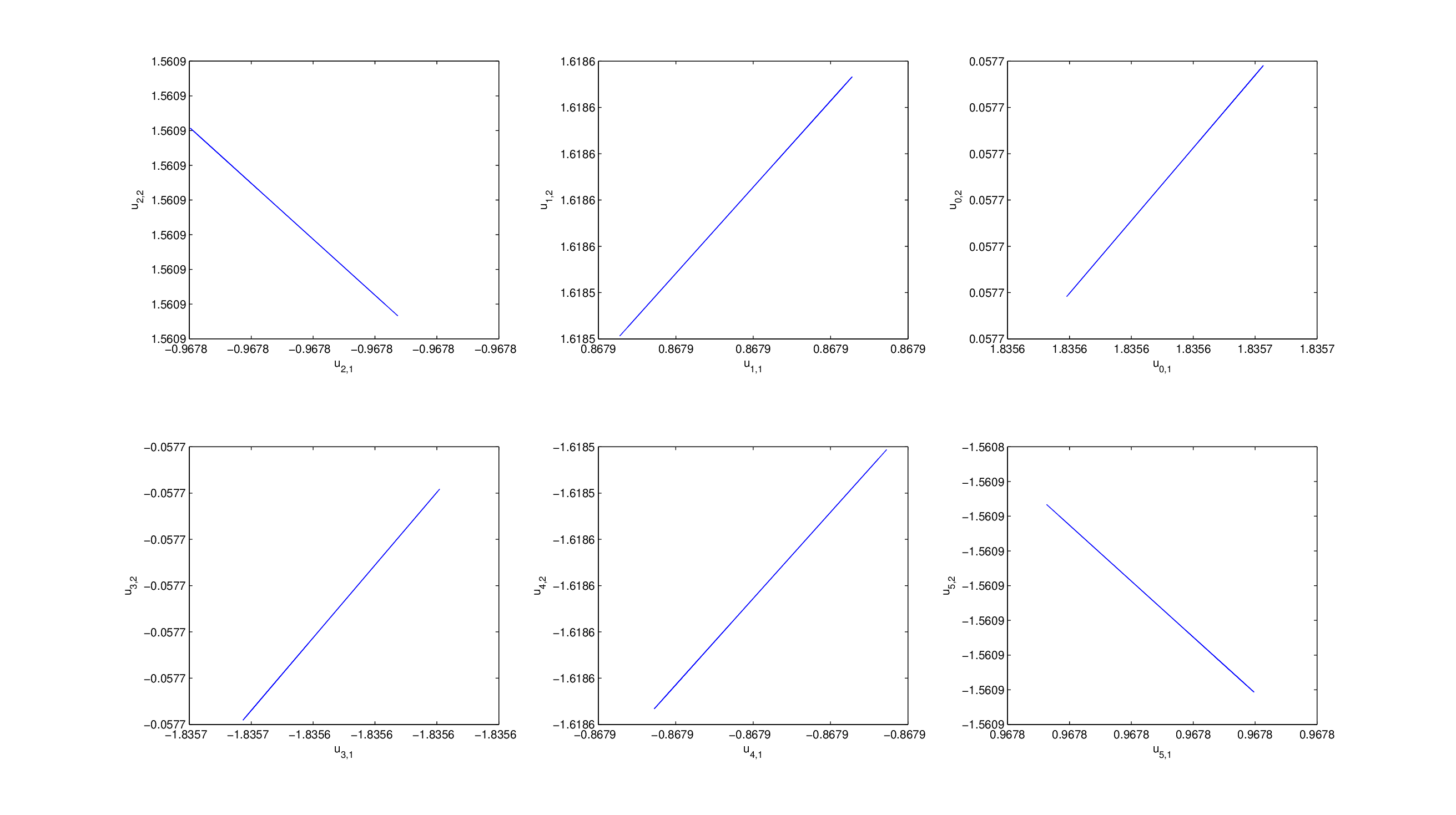}
}\caption{Relative positions of all $n=6$ particles with $\lambda^2=\frac{l^2}{\mu}$, $l=1$  and $\mu$ near the eigenvalue $\mu=43.00585474$ of $\nabla^2 V(u^o)$}
\label{relative-u-43}
\end{center}
\end{figure}

\begin{figure}[H]
\begin{center}
\scalebox{0.45}{
\includegraphics[angle=0]{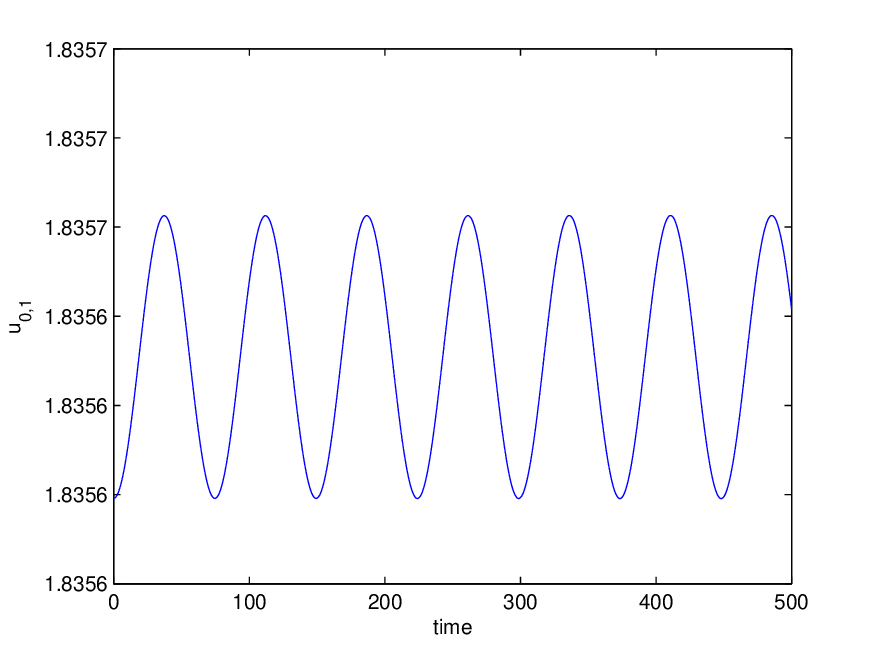}
}
\scalebox{0.45}{
\includegraphics[angle=0]{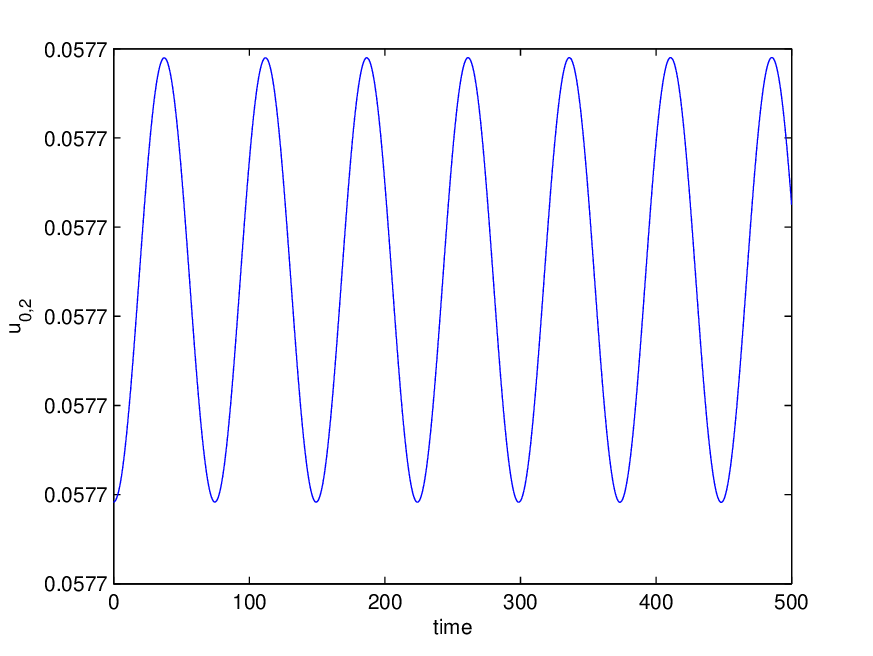}
}
\caption{Motion of $u_0$ with $\lambda^2=\frac{l^2}{\mu}$, $l=1$  and $\mu$ near the eigenvalue $\mu=43.00585474$ of $\nabla^2 V(u^o)$}
\label{fig-u-43}
 \end{center}
\end{figure}

\begin{figure}[H]
\begin{center}
\scalebox{0.4}{
\hspace*{-10em}
\includegraphics[angle=0]{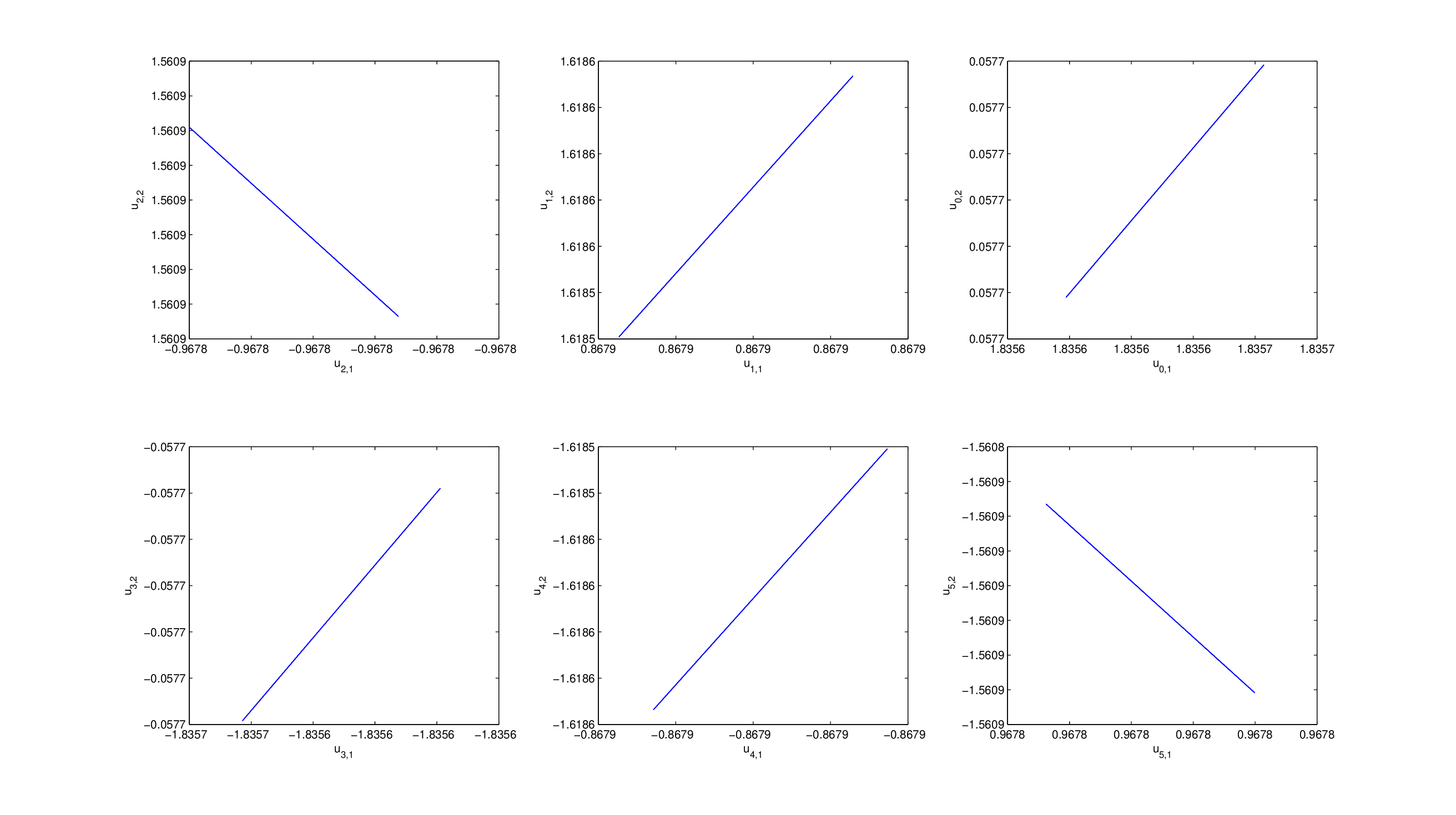}
}\caption{Relative positions of all $n=6$ particles with $\lambda^2=\frac{l^2}{\mu}$, $l=1$  and $\mu$ near the eigenvalue $\mu=19.58406142$ of $\nabla^2 V(u^o)$}
\label{relative-u-19}
\end{center}
\end{figure}

\begin{figure}[H]
\begin{center}
\scalebox{0.5}{
\hspace*{-0em}
\includegraphics[angle=0]{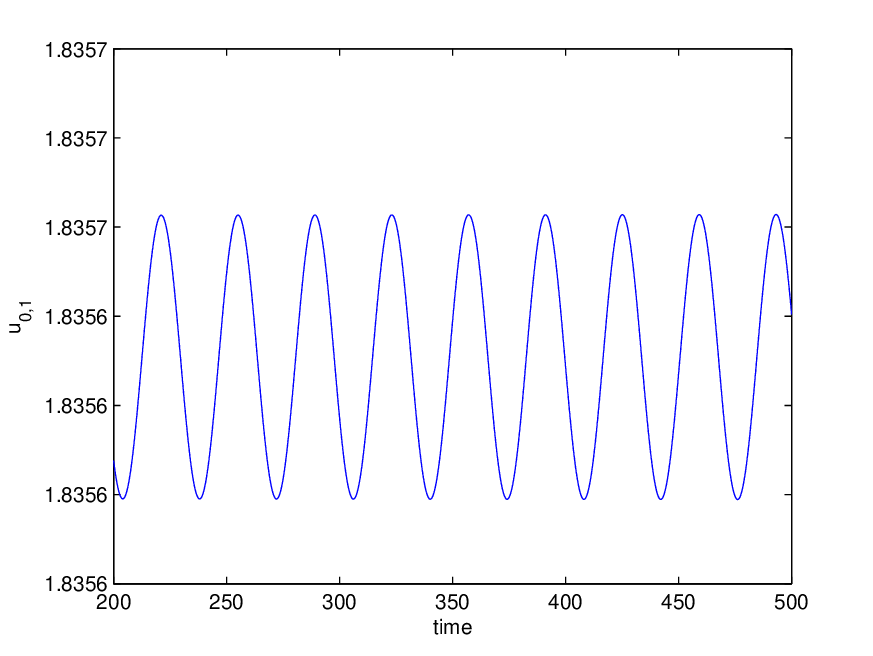}}
\scalebox{0.5}{
\hspace*{-5em}
\includegraphics[angle=0]{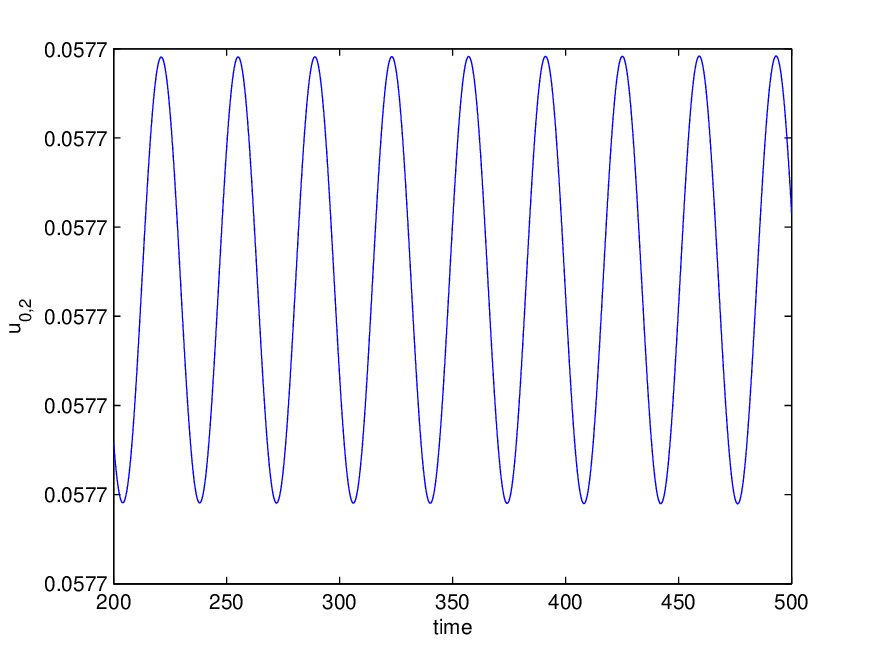}
}\caption{Motion of $u_0$ with $\lambda^2=\frac{l^2}{\mu}$, $l=1$  and $\mu$ near the eigenvalue $\mu=19.58406142$ of $\nabla^2 V(u^o)$}
\label{fig-u-19}
\end{center}
\end{figure}

\begin{figure}[H]
\begin{center}
\scalebox{0.45}{
\hspace*{-16em}
\includegraphics[angle=0]{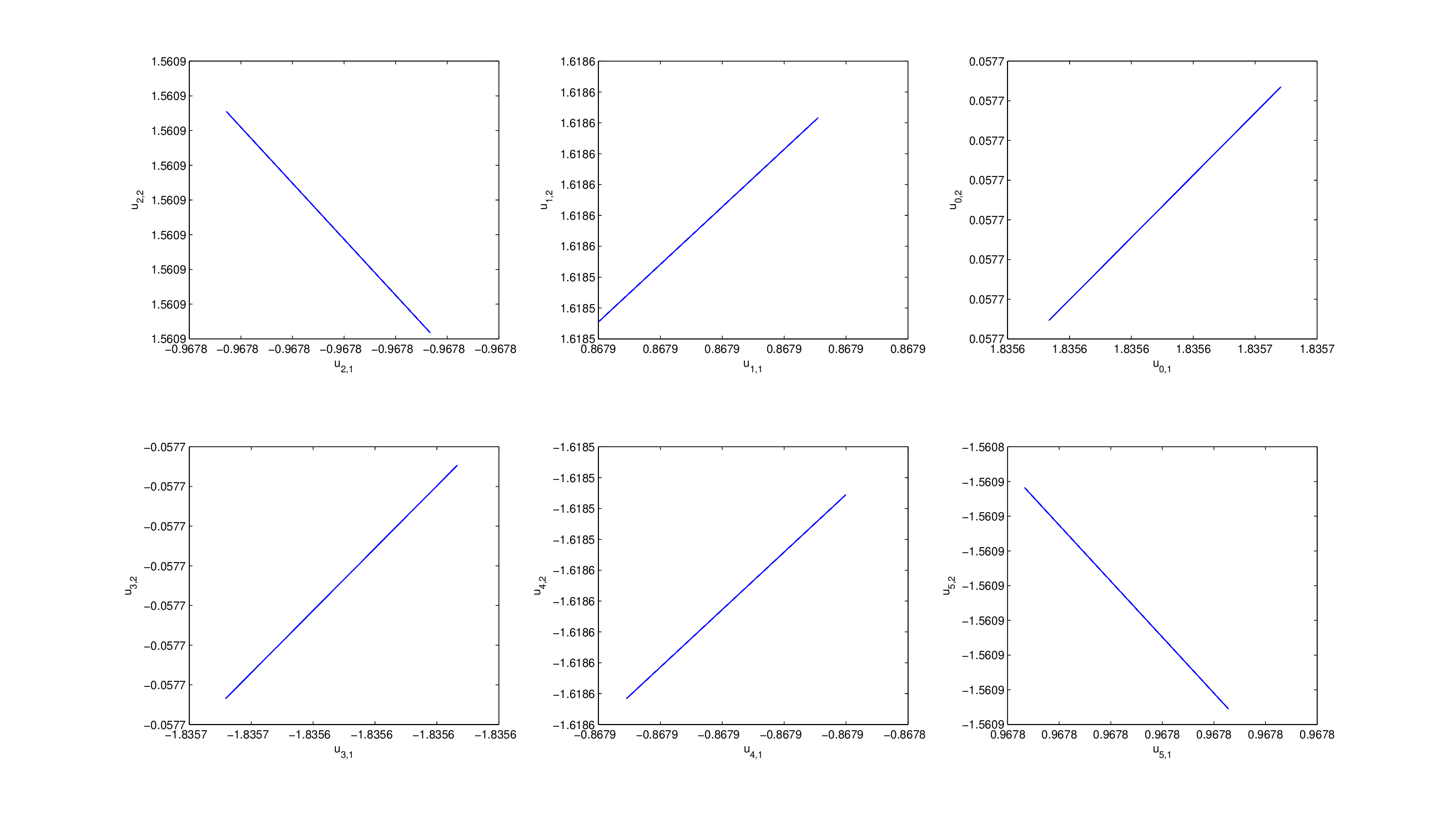}
}\caption{Relative positions of all $n=6$ particles with $\lambda^2=\frac{l^2}{\mu}$, $l=1$  and $\mu$ near the eigenvalue $\mu=7.633501334$ of $\nabla^2 V(u^o)$}
\label{relative-fig-u-7}
\end{center}
\end{figure}

\begin{figure}[H]
\begin{center}
\scalebox{0.5}{
\hspace*{-2em}
\includegraphics[angle=0]{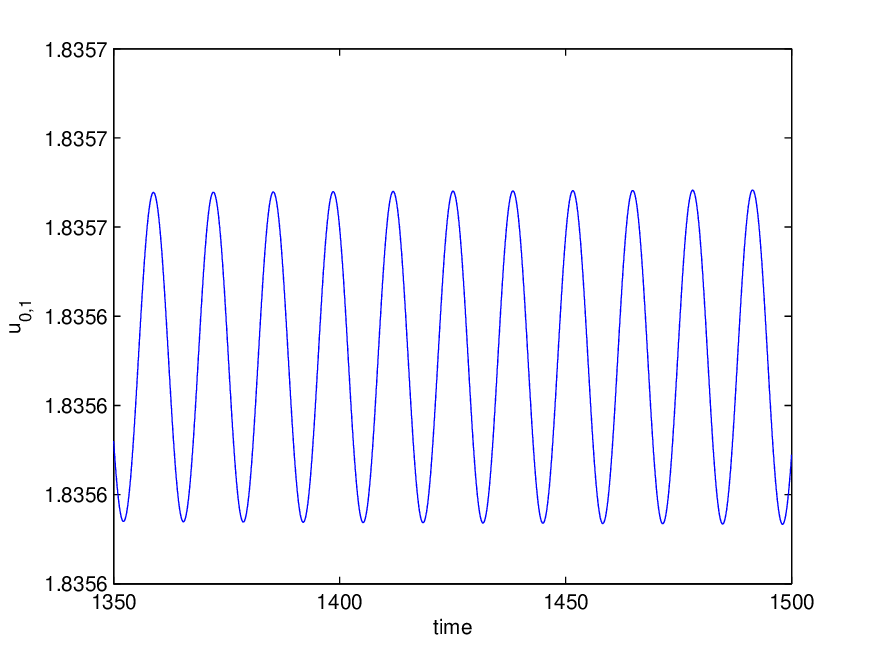}}
\scalebox{0.5}{
\hspace*{0em}
\includegraphics[angle=0]{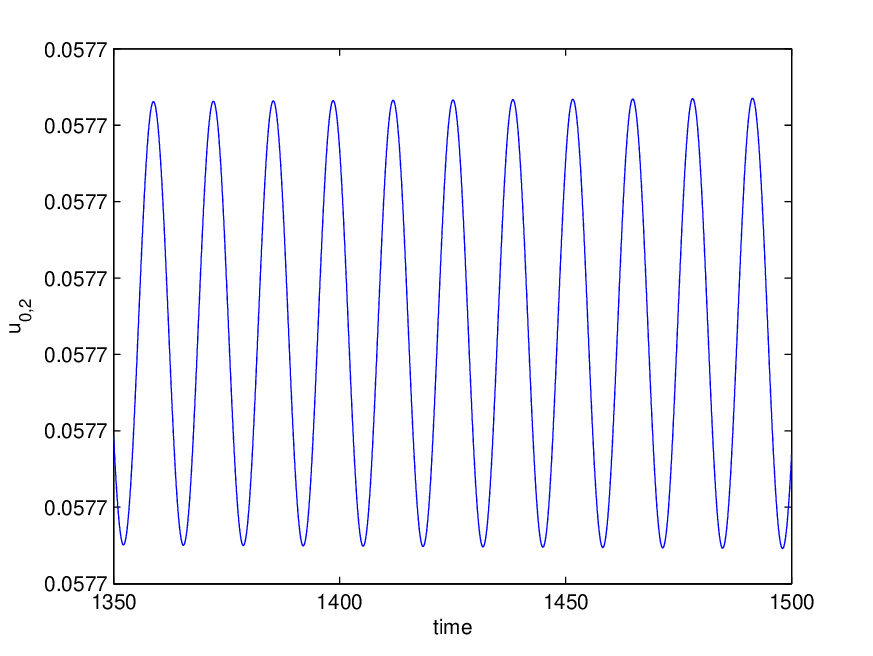}
}\caption{Motion of $u_0$ with $\lambda^2=\frac{l^2}{\mu}$, $l=1$  and $\mu$ near the eigenvalue $\mu=7.633501334$ of $\nabla^2 V(u^o)$}
\label{fig-u-0-7}
\end{center}
\end{figure}

\begin{figure}[H]
\begin{center}
\scalebox{0.45}{
\hspace*{-16em}
\includegraphics[angle=0]{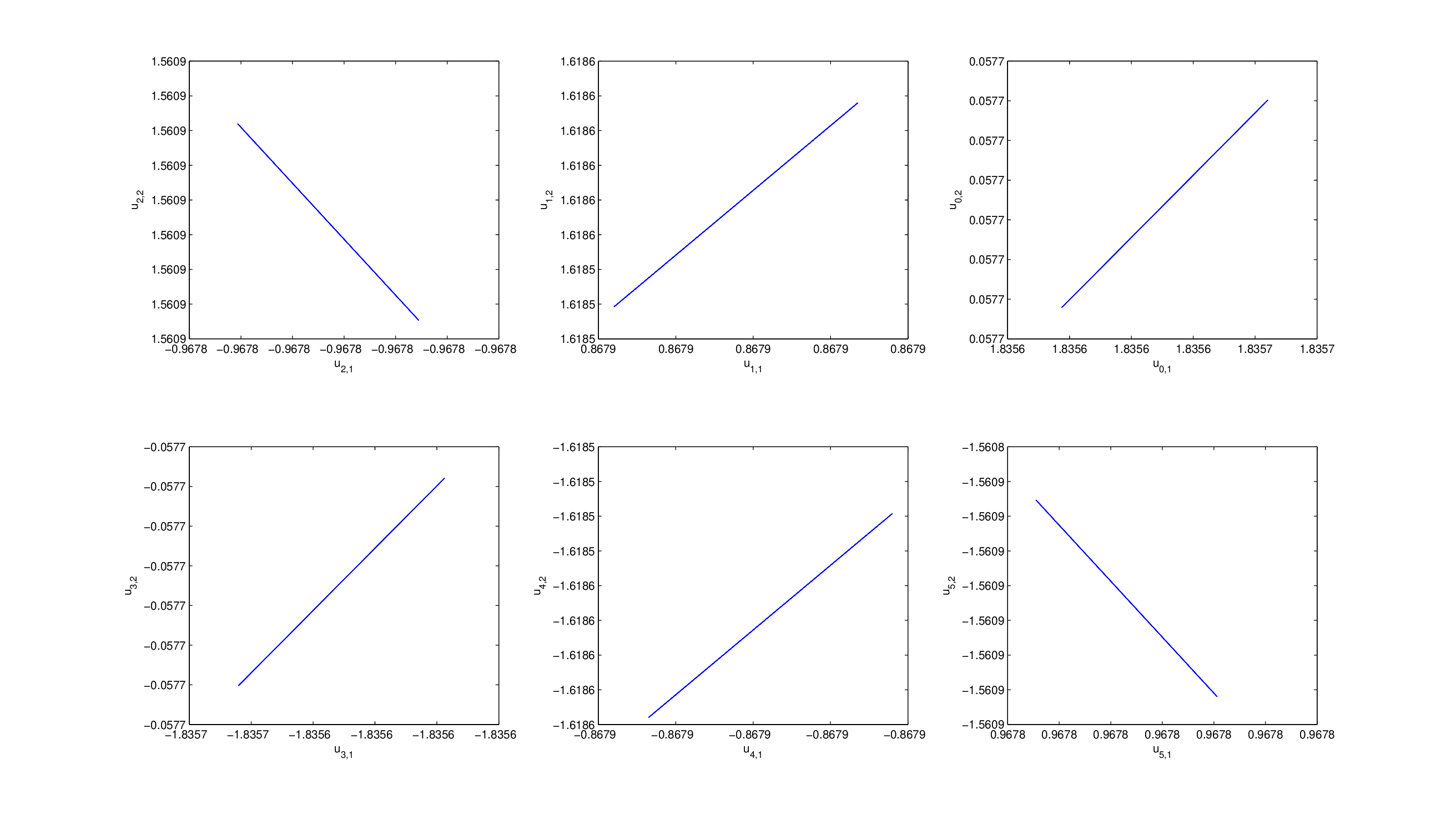}
}\caption{Relative positions of all $n=6$ particles with $\lambda^2=\frac{l^2}{\mu}$, $l=1$  and $\mu$ near the eigenvalue $\mu=11.42339623$ of $\nabla^2 V(u^o)$}
\label{relative-fig-u11}
\end{center}
\end{figure}

\begin{figure}[H]
\begin{center}
\scalebox{0.5}{
\hspace*{-2em}
\includegraphics[angle=0]{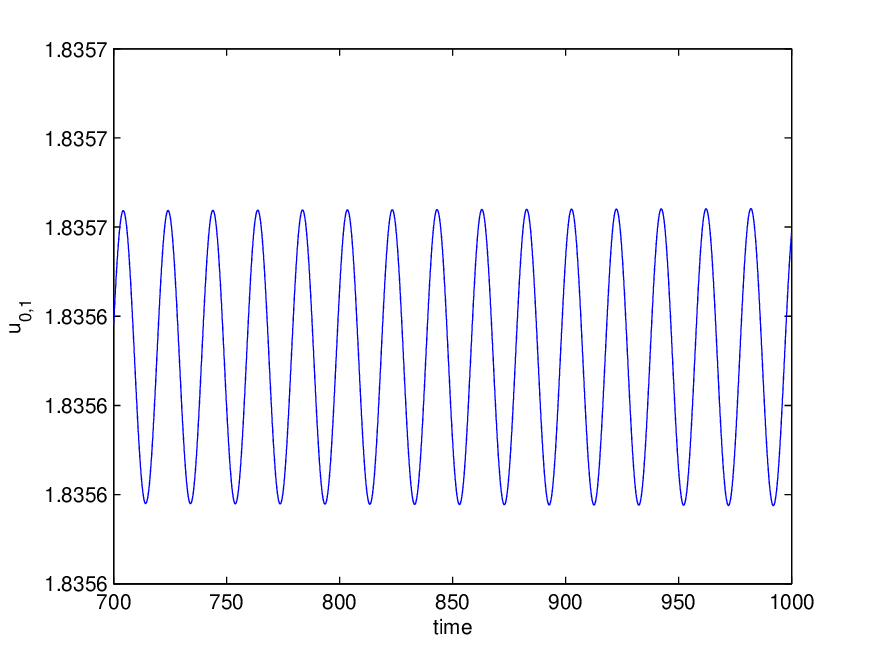}}
\scalebox{0.5}{
\hspace*{0em}
\includegraphics[angle=0]{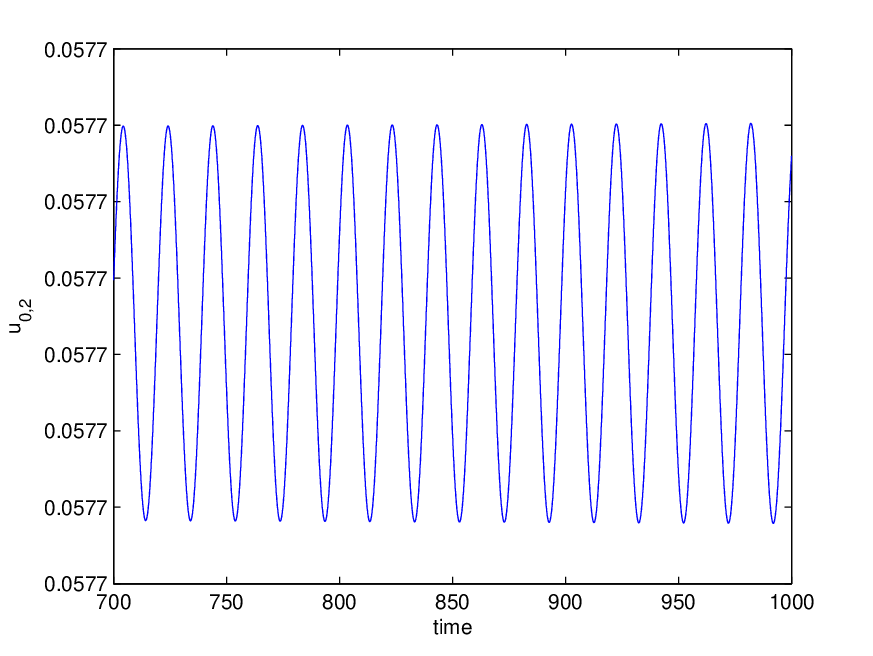}
}\caption{Motion of $u_0$ with $\lambda^2=\frac{l^2}{\mu}$, $l=1$  and $\mu$ near the eigenvalue $\mu=11.42339623$ of $\nabla^2 V(u^o)$}
\label{fig-u-11}
\end{center}
\end{figure}
\section{Concluding Remarks}
 \vs
In this paper, we analyzed a system  \eqref{eqn01} with $n$ particle in the plane $\br^2$ admitting dihedral spatial symmetries. More precisely, 
we use the   method of gradient equivariant degree \cite{Geba,survey,FRR,RR} to investigate the existence of periodic solution to \eqref{eqn01}, where $\bV$ is the Lennard-Jones and Coulomb potential, around an equilibrium admitting dihedral $D_n$ symmetries. The dynamics of system \eqref{eqn01} can be very complicated with a large number of different periodic solutions exhibiting various spatio-temporal symmetries. The equivariant degree provides equivariant invariants  for  system \eqref{eqn01} allowing a complete symmetric topological classification of the emanating (or bifurcating) branches  of periodic solutions from a given equilibrium state. First, the critical periods  $p_{jl}>0$, which are the limit periods for those bifurcation branches can be identified from the so called {\it critical set} $\Lambda:=\{ \lambda_{jl}=\frac {l^2}{\mu_j}, \; l\in \mathbb N,\; \mu_j\in \sigma(\nabla^2 \bV(u^o))\}$, where $\sigma(\nabla^2 \bV(u^o))$ denotes the set of eigenvalues of the Hessian $\nabla^2 \bV(u^o)$, and the symmetries of topologically possible solutions to \eqref{eqn01} can be identified from the equivariant invariants $\omega_{\bf G} (\lambda_{jl})$. The explicitly computed Hessian $\nabla^2 \bV(u^o)$ facilitated the formulation of  general results for  dihedral molecular configurations. 

We developed a method using the isotypical decomposition of the phase space combined with block decompositions and the usual complex operations in order to represent $\nabla^2 \bV(a)$ as a product of simple $2\times 2$-matrices. Therefore, the spectrum $\sigma(\nabla^2 \bV(a))$ is explicitly computed and these computations do not depend on a particular form of the potential $\bV$. In addition, we provided an exact formula for computation of the equivariant invariants  $\omega_{\bf G} (\lambda_{jl})$. We should also mention that for larger groups $D_n$, the actual computations of $\omega_{\bf G} (\lambda_{jl})$ can be quite complicated but still possible with the use of computer software. Such software was already developed  for several types of groups $G=\Gamma\times O(2)$ and it is available at \cite{Pin}.

We reamrk that elements $\lambda_{jl}$ of the critical set $\Lambda$ correspond to the values of transitional frequencies. The equivariant invariant $\omega_{\bf G} (\lambda_{jl})$ provides a full topological classification of symmetric modes corresponding to the branches of molecular vibrations emerging from the equilibrium  state at the critical frequency $\lambda_{jl}$. This method can be applied to create, for a molecule with dihedral symmetries, an {\it atlas} of topologically possible symmetric modes of vibrations, the collection of actual distinct  molecular vibrations (related the maximal symmetric types) emerging from the equilibrium state and the corresponding limit frequencies. 
\vs

\end{document}